\documentclass[a4paper,12pt]{article}

\usepackage{amsmath,amssymb,amsfonts}
\usepackage{amsthm}
\usepackage{mathtools}
\usepackage{xcolor}
\usepackage{graphicx}
\usepackage{subcaption}
\usepackage{blkarray}
\usepackage{tabularx}
\usepackage{algorithm}
\usepackage{csquotes}
\usepackage{caption}
\usepackage{subcaption}
\usepackage[margin=1in]{geometry}
\usepackage{hyperref}
\usepackage{algorithm}
\usepackage{float}
\usepackage{algpseudocode}
\usepackage{subcaption}
\usepackage[font=small, labelfont=bf, width=0.85\textwidth, justification=centering, skip=10pt]{caption}
\usepackage{placeins}
\graphicspath{{Fig/}}

\newtheorem{remark}{Remark}

\newtheorem{definition}{Definition}
\newtheorem{corollary}{Corollary}
\newtheorem{lemma}{Lemma}

\newtheorem{proposition}{Proposition}

\numberwithin{equation}{section}

\title{\large A hybrid isogeometric and finite element method: NURBS-enhanced finite element method for hexahedral meshes (NEFEM-HEX)}

\usepackage{authblk}

\author{
  \small Duygu Sap\thanks{Corresponding author: \href{mailto:duygu.sap@warwick.ac.uk}{duygu.sap@warwick.ac.uk}}%
  \thanks{The substantial portion of this work was carried out while the author was affiliated with the University of Oxford.}
}
\affil{\small CAMaCS, Mathematics Institute, University of Warwick,\\ Coventry, CV4 7AL, United Kingdom.}
\date{}
\begin{document}

\maketitle
\begin{abstract}
In this paper, we present a NURBS-enhanced finite element method that integrates the NURBS-based boundary representation of a geometric domain into a standard finite element framework for hexahedral meshes. We decompose an open, bounded, convex three-dimensional domain with a NURBS boundary into two parts, define NURBS-enhanced finite elements over the boundary layer, and use piecewise-linear Lagrange finite elements in the interior region. We introduce a special quadrature rule and a stable interpolation operator for the NURBS-enhanced elements. We discuss how the h-refinement in finite element analysis and the knot insertion in isogeometric analysis can be utilized in the refinement of the NURBS-enhanced elements. To illustrate an application of our methodology, we utilize a generic weak formulation of a second-order linear elliptic boundary value problem and derive a priori error estimates in the $H^{1}$ norm.  In addition, we use the Poisson problem as a model problem and provide numerical results that support the theoretical results. The proposed methodology combines the efficiency of finite element analysis with the geometric precision of NURBS, and may enable more accurate and efficient simulations over complex geometries. 
\end{abstract}

\maketitle\section{Introduction}
Numerical methods that offer accurate and exact geometric representations of the computational domain are crucial in solving real-world problems as the majority of engineering problems require working with domains that have curved or complex boundaries.
The standard finite element technique used in domains with curved boundaries is the isoparametric finite element method\cite{brenner}. The isoparametric finite element method utilizes the piecewise polynomial functions that can be used to parametrize curved boundaries in analysis. Although they do not represent the exact geometry of the physical domain, they yield high-order approximations mitigating the geometric error \cite{neilan2023stable}. They were first introduced by Irons et al \cite{irons} in two dimensions. Lenoir, then, described a practical procedure for the triangulation of arbitrary n-dimensional, regular, bounded domains by isoparametric simplicial elements and derived general error estimates for the finite element solutions of second-order elliptic problems\cite{lenoir}. There is a vast amount of publications that involve isoparametric finite  elements as of today. \cite{ciarlet1972, scott_1973, zlamal} are among the classical works on isoparametric finite elements. Nevertheless, polynomials have limited geometric representation power, and the approximate geometry defined by isoparametric finite element methods may cause a significant loss in the overall accuracy of the method. Thus, in many applications, geometric errors introduced by the isoparametric mappings deteriorate the accuracy of the numerical solution, and an exact representation of the geometry is necessary for achieving optimal accuracy for a given spatial discretization \cite{hughes}. 
In the p-version of the finite element method \cite{duster}, a coarse mesh with elements that usually have large aspect ratios and can represent the boundary exactly are utilized. The mesh remains unchanged throughout the analysis while the polynomial order of the approximation is increased in order to accurately approximate the solution. This may result in an oscillatory behavior near discontinuities -- which is described as the Runge's phenomenon\cite{runges}. Also, the bijectivity of the geometric map may be very sensitive to the geometric location of the nodes for $p\ge 2$; therefore, in practice, finite element meshing tools do not offer geometric approximations higher than cubic order if the nodes are regularly distributed, and low-order finite elements such as linear, quadratic, or cubic finite elements are widely used when a regular node distribution is preferred \cite{ciarlet1972, lenoir}. \\ 
\indent Unlike polynomials, Non-Uniform Rational B-splines (NURBS) can be used to precisely represent and design both standard analytic shapes, such as conics, quadrics, surfaces of revolution, and free-form curves and surfaces \cite{nurbssurvey}. 
Isogeometric analysis (IGA), which is introduced by Hughes et al \cite{hughes} as an alternative to finite element analysis (FEA), integrates the NURBS boundary representation into the isoparametric finite element approach by using the NURBS bases used for the geometry in analysis. Thus, it enables the use of the exact representations of curved boundaries. It also prohibits potential oscillations near discontinuities observed in p-FEM via the variation-diminishing property of NURBS \cite{nurbsbk}. 
IGA primarily focuses on the exact geometric representation of the boundary. In some works on IGA, B-splines are used in analysis although NURBS are used to exactly represent the geometry. We refer the reader to \cite{kiendl2009,bazilevs2010,buffa2016,falco2011} for a short list of applications and a comprehensive introduction to IGA. IGA has been widely used in many engineering applications since its emergence. It plays a crucial role in CAD-CAE interoperability since FEA is the most commonly used analysis technique in engineering and NURBS is the standard boundary representation used in the majority of CAD software. However, the interoperability of CAD and CAE remains a challenging problem as CAD modellers provide only the parametrizations of the boundaries of geometric objects as collections of manifolds, that is, a collections of 2D objects in 3D, while
the approximation spaces need to be defined over 3D volumetric objects representing the computational domains. 
To date, various approaches have been followed to integrate CAD and FEA
\cite{precursors, cavendish, gmsolid}.
For example, in the FE-IGA approach, the surfaces generated by the Computer-Aided Geometric Design (CAGD) tools are extended to generate volumes. See \cite{fe-iga1, fe-iga2} for some applications of this approach. We remark that generating volumetric description from boundary representations is an open problem.  On the other hand, while working with the partial differential equations that admit formulations involving only boundary integrals, isogeometric boundary element method (IGABEM) can be applied. IGABEM avoids generating volumetric descriptions, thus, reduces the dimensionality of the problem from 3D to 2D. It also enables the treatment of problems in infinite exterior domains. 
However, we note that the majority of partial differential equations that come up in physics and engineering applications do not admit formulations in terms of only boundary integrals. Thus, IGABEM can be applied to only a specific set of  problems. See \cite{igabem,igabem1,igabem2}  for some applications of the IGABEM.
Despite the lack of CAD-CAE interoperability, various initiatives have been taken to integrate IGA into commercial FEA software. See \cite{hartmann2016isogeometric,khakalo2017isogeometric,rypl2012finite} for examples. Regardless, while using commercial FEA packages for IGA, generating and visualizing NURBS geometries
can be complex and is not fully-supported \cite{iga-simplified}. Such computational challenges are out of the scope of this work.\\ 
\indent NURBS-enhanced discrete element and NURBS-enhanced finite element methods (NEFEM) comprise another set of methods that utilize the NURBS representation of the boundary in analysis. These methods have been applied to various problems in fluid dynamics and contact mechanics \cite{make,corbett}. Keng-Wit presented a NURBS-enhanced discrete element method for contact dynamics applications in \cite{nedem}. Sevilla et al \cite{ruben} introduced NURBS-enhanced finite element methods (NEFEM) over triangular and tetrahedral meshes.  NEFEM enables utilizing the efficiency of finite element methods as well as the ability of NURBS to represent conic shapes exactly.  It also allows avoiding the use of 3D NURBS required by IGA completely \cite{ginnis,cheng}. \\
 \indent   In this paper, we construct a NURBS-enhanced finite element method for hexahedral meshes. 
This method comprises the piecewise-linear Lagrange finite element method in 3D and a hybrid finite element method which utilizes both the NURBS basis functions and the linear Lagrange finite element basis functions in 2D. The hybrid finite element method and the piecewise-linear Lagrange finite element method are used over different regions of the domain. We use the blending function method introduced by Gordon and Hall\cite{blend} and follow the conventional approach employed in FEA and IGA that relies on the use of a reference element and locally-defined basis functions. Thus, our approach is fundamentally different than the ones used in \cite{ruben1,ruben}. We introduce a novel quadrature rule and interpolation operator for the hybrid (that is, the NURBS-enhanced) finite elements. We assume that the mesh is shape-regular and does not have any singularities. We discuss how the h-refinement in finite element analysis and the knot insertion in isogeometric analysis can be used simultaneously over the NURBS-enhanced elements. To illustrate an application of our methodology and analyze the convergence properties of our approach, we utilize a generic weak formulation of a linear second-order elliptic boundary value problem and derive a priori error estimates in the $H^1$ norm. In addition, we present the results of numerical experiments for the Poisson problem that support the theoretical results.\\ 
\indent We note that one would still need to address some of the challenges IGA suffers from while working with NEFEM due to the involvement of NURBS within NEFEM. For example, the tensor-product nature of the basis functions prevents local mesh refinement. One approach to overcome this issue is to consider hierarchical B-Spline \cite{evans2020hierarchical} (or hierarchical NURBS\cite{hierNURBS}) basis instead of a tensor-product basis \cite{evans2020hierarchical}. Another approach is based on using T-Splines \cite{tspline,buffa2014isogeometric}, which allow for local refinement via T-junctions but result in a loss of smoothness in the approximation. Other drawbacks of NEFEM are concerned with the treatment of singularities that decreases the convergence rate as in the case of IGA and p-FEM, complicated mesh generation, and the treatment of trimmed or singular NURBS that are widely used in CAD\cite{ruben}. We assume that the mesh is shape-regular and does not have any singularities and only consider global mesh refinement. Thus, we leave these highly-technical issues out of the scope of this paper.\\   
\indent The outline of this manuscript is as follows: In Section~\ref{sec:method}, we first provide some preliminary information about B-splines and NURBS, then a detailed description of our methodology; in Section~\ref{sec:stability}, we introduce the interpolation operators and list or derive their stability and approximation properties;
in Section~\ref{sec:modelprb}, we address the Poisson problem as a model problem, derive a priori error estimates, and provide results of numerical experiments; 
in Section~\ref{sec:conclusion}, we provide an overview of our results and state prospective research directions.
\section{Methodology} \label{sec:method}
Let $\Omega \subset \mathbb{R}^3$ be an open, bounded, convex curved domain with a Lipschitz boundary and suppose its boundary, $\partial \Omega$, is described by a single NURBS patch denoted by $\mathcal{P}$. We define $\bar{\Omega}:=\Omega \cup \partial \Omega$ and let $\mathcal{T}_h$ be the hexahedral discretization of $\bar{\Omega}$, where  $h=\max\limits_{\mathcal{Q}\in \mathcal{T}_h}\{ h_{\mathcal{Q}}\}$ is the global mesh size. We classify the hexahedral elements as boundary and interior elements by defining $\mathcal{T}_h^{(b)}:=\{Q\in \mathcal{T}_h: Q\cap \partial \Omega\neq \emptyset\}$ and $\mathcal{T}_h^{(i)}:=\mathcal{T}_h\setminus \mathcal{T}_h^{(b)}$. (See Figure~\ref{fig:adjacent-elements} for a pair of adjacent boundary layer and interior elements.)  Then, we use these two sets of elements to define the boundary layer $\Omega_{\mathcal{B}}$ and interior region $\Omega_{int}$ of the domain as follows:
\begin{align*}
   &\Omega_{\mathcal{B}}:=\bigcup\limits_{\substack{Q \in \mathcal{T}_h^{(b)}}} Q, &\Omega_{int}:= \bigcup\limits_{\substack{Q \in \mathcal{T}_h^{(i)}}} Q
\end{align*}
We assume that the interface between these two regions is planar, that is, $(\Omega_{int} \cap \Omega_{\mathcal{B}})$ is a polyhedral surface although $\Omega_{\mathcal{B}}$ consists of elements with curved faces on $\partial \Omega$. 
\begin{figure}[H]
    \centering
    \includegraphics[scale=0.7]{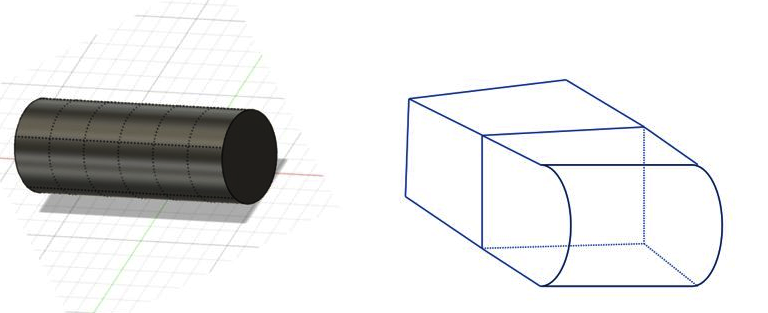}
  \caption{A cylinder (on the left) depicted with elements in $\mathcal{T}_h^b$ and $\mathcal{T}_h^i$ (on the right).}
    \label{fig:adjacent-elements}
\end{figure}
\subsection{Preliminaries}\label{sec:prelims}
In this subsection, we briefly introduce B-splines and NURBS along with the meshes associated with them. For more details, we refer the interested reader to \cite{schumaker, nurbsbk,veiga-buffa},
and \cite{farin}. 
\subsubsection{B-Splines}
B-splines are piecewise polynomial functions that satisfy certain regularity conditions. They are defined via knot vectors that partition a reference domain. By convention, their reference domain is defined as the unit interval in 1D and the unit hypercube in $d$-dimensions.\\ 
\noindent To define a B-spline of degree $p$ in 1D, we use a knot vector $\Sigma=[\eta_1,\dots, \eta_{k}]$ with components (a.k.a. \textit{knots}) satisfying $0\le \eta_{1}\le \eta_2 \le \dots \le  \eta_{k}\le 1$. 
\noindent The regularity of the B-spline at a knot $\eta_i\in \Sigma$ is given by $r_i:=p-m_i$ where $m_i\ge 1$ denotes the multiplicity of $\eta_i$ in $\Sigma$. In this manuscript, we assume that $\Sigma$ is an open knot vector, that is, $r_1 = r_{k} = p+1$, and the splines are at least continuous at the knots, that is, {$m_i \le p$}. B-spline basis functions are defined recursively via the Cox-de Boor formula\cite{deboor} as follows:
 \begin{align}\label{eq:bspline-basis}
B_i^0(\eta)&:=\begin{array}{cc}
  \{ \quad 
    \begin{array}{cc}
      1,& \eta_i\le \eta < \eta_{i+1}, \\
      0,&  otherwise
    \end{array}
\end{array}
\\
B_i^p(\eta)&=\frac{\eta- \eta_i}{\eta_{i+p}-\eta_i} B_i^{p-1}(\eta)+\frac{\eta_{i+p+1}-\eta}{\eta_{i+p+1}-\eta_{i+1}} B_{i+1}^{p-1}(\eta)\notag 
\end{align}

\noindent To define a B-spline object such as a B-spline curve or surface we need another set of parameters called the \textit{control points}.  Unlike the knots, control points lie in the physical domain. In IGA, they are treated as the equivalent of finite element nodes although splines are generally not interpolatory at the control points. From a given set of control points $\{c_i\}_{i=1}^n$ in $\mathbb{R}^3$, a B-spline curve of degree $p$ is constructed by the parametric map $C^p: [0,1]\to \mathbb{R}^3$ defined as:
$$C^p(u)= \sum\limits_{i=1}^n c_i B_{i}^p(u).$$

\subsubsection{Non-Uniform Rational B-Splines (NURBS)}
\noindent NURBS are derived from B-splines and have rational forms. In 1D, NURBS basis functions are defined via the B-spline basis functions given by \eqref{eq:bspline-basis} and a set of scalars $\{w_i\}_{i=1}^n$ called the \textit{weights}. For example, a NURBS basis function of degree $p$ is defined as follows:
\begin{align*}
R_i^p(u)&=\frac{w_i\ B_i^p(u)}{W(u)}, 
\end{align*}
\noindent where $W(u)=\sum\limits_{i=1}^n w_i B_i^p(u)$ is called the \textit{weight function}.\\ 
Similar to a B-spline curve, a NURBS curve can be generated using a given set of control points $\{c_i^w\}_{i=1}^n$ in $\mathbb{R}^3$ via the parametric map $C_{\mathcal{N}}:[0,1]\to \mathbb{R}^3$ defined as:
\begin{align*}
C_{\mathcal{N}}(u)=\sum\limits_{i=1}^n c_i^w\ R_i^p(u),
\end{align*}
where the superscript $w$ indicates the association of weights with the control points - which is one of the features of NURBS that distinguish them from B-splines.

\noindent In higher-dimensions, NURBS basis functions can be obtained via the tensor-product of the one-dimensional NURBS basis functions, as in the case of B-splines. For example, in two-dimensions, a NURBS basis function can be derived as follows:
\begin{align*}
R_{ij}^{pq}(u,v):= R_i^p(u)\otimes R_j^q(v)= \frac{w_{ij}\ B_i^p(u)\ B_j^q(v)}{\sum\limits_{i=1}^n \sum\limits_{j=1}^m w_{ij}B_i^p(u)\ B_j^q(v) },
\end{align*}
\noindent where $w_{ij}=w_i w_j$ denote the weights. A  NURBS surface associated with the control points $\{C_{ij}\}_{i,j=1}^{n,m}$ then can be defined by a parametric map $S: [0,1]^2\to \mathbb{R}^3$ as follows:
\begin{align}\label{eq:nurbs-map}
    S(u,v)=\sum\limits_{i,j=1}^{n,m} C_{ij}R_{ij}^{pq}(u,v),
\end{align}
\subsubsection{Meshes}\label{sec:meshes}
A NURBS surface can be decomposed into multiple patches - each of which is an image of a grid of rectangular elements in the reference domain under a specific NURBS map.  
A NURBS patch can be defined as a tensor product of two NURBS curves; thus, its knot net can be derived from the tensor product of the knot vectors of two curves. 
Suppose that $\mathcal{P}$ is defined via a knot-net obtained from the tensor product of the open knot vectors $K_1^o=\{\eta^1_1,\dots,\eta^1_{n_1+p_1+1} \}$ and $K_2^o=\{\eta^2_{1},\dots,\eta^2_{n_2+p_2+1} \}$, where $n_i$ and $p_i$ denote the number of control points and polynomial degree used to define the $i^{th}$ NURBS curve, respectively. Let $K_1$ and $K_2$ be the sets of non-repeating knots  (a.k.a. \textit{breaking points}) obtained from $K_1^o$ and $K_2^o$, respectively, and denote by $N_i$ the cardinality of $K_i$. \\
A \textit{Bezier mesh} $\mathcal{T}_S$ associated with these knot vectors is then defined as: 
\begin{align*}
    \mathcal{T}_S&:=\{\hat{I}_{ij}:= [\alpha_{i},\alpha_{i+1}]\otimes [\beta_{j},\beta_{j+1}]:\ \alpha_i \in K_1,\\ 
   & \beta_j\in K_2,\ |\hat{I}_{ij}| \neq \emptyset,\ i=1,\dots,N_i,\ j=1,\dots, N_j\},
\end{align*}
\noindent where $|\hat{I}_{ij}|$ denotes the measure of $\hat{I}_{ij}$.
\noindent The image of $\mathcal{T}_S$ under a NURBS parametic map $S$ defined as in \eqref{eq:nurbs-map} yields the \textit{physical Bezier mesh} over $\mathcal{P}$ defined as follows: 
$$\mathcal{T}_S^p:=\{I_{ij}: S(\hat{I}_{ij}),\ \forall \hat{I}_{ij} \in \mathcal{T}_S\}.$$

 \noindent Another important mesh associated with a NURBS surface is the \textit{Greville mesh}, which plays a crucial role in our methodology. A Greville point (a.k.a. \textit{Greville abcissae}) associated with a B-spline basis function $B_{i}^k$, where $k\in \{1,2\}$ is the dimension index, is given by
 
\begin{equation}\label{eq:grevpts}
\gamma_i^k=\frac{\eta^k_{i+1}+\dots+ \eta^k_{i+p_k}}{p_k}.
\end{equation}
\noindent $\gamma_i^k$ decomposes the identity in the $k^{th}$ direction in the B-spline basis, i.e. $\eta= \sum\limits_{i=1}^{n_k}\gamma_i^k B^k_i(\eta)$\cite{veiga-buffa}.\\
 \noindent Greville points are used to define the Greville mesh as follows: 
\begin{equation}\label{eq:greville-mesh}
    \mathcal{G}_h:=\{(\gamma_i^1,\gamma_j^2): 1\le i \le n_1,\ 1\le j\le n_2\}.
\end{equation}

 \noindent Since splines are not interpolatory, we will use the Greville points while defining the function spaces and the degrees of freedom. We assume that the multiplicity of the internal knots are less than or equal to $p_k$ in the $k^{th}$ direction, therefore, the Greville points are distinct and form a partition of the interval. Then, there exist piecewise bilinear functions $\{\phi_{ij}:=\phi_{i}^1\phi_{j}^2\}_{i,j=1}^{n_1,n_2}$ that are dual to these Greville points $\{\gamma_{mn}:=(\gamma_m^1, \gamma_n^2)\}$, that is, $\phi_{ij}(\gamma_{mn})=\phi^1_{i}(\gamma_m^1)\phi_{j}^2(\gamma_n^2)=\delta_{im}\delta_{jn}$, where $\phi^k_{i}$ indicates the dual function corresponding to the $i^{th}$ Greville point in the $k^{th}$ parametric direction, $\gamma^k_i$, and $\delta$ denotes the Kronecker delta \cite{veiga-buffa}.
 \noindent Using the functions $\{\phi_{ij}\}$, one can define the \textit{control mesh} indirectly from the Greville mesh as follows:
\begin{equation}\label{eq:control-mesh}
    \mathbb{C}:=\sum\limits_{i.j=1}^{n_1, n_2} C_{ij} \phi_{ij},\ \text{with}\ \mathbb{C}(\gamma_{ij})= C_{ij}.
\end{equation}
As \eqref{eq:control-mesh} implies, the control mesh is the piecewise linear interpolation of the control points associated with the NURBS surface. See Figure~\ref{fig:domains} for an illustration of a NURBS patch with its control mesh and knot-net on the reference domain. Moreover, if splines (or NURBS) are used to describe a function, then the control field corresponding to this function is a piecewise linear finite element function defined
on the Greville mesh \cite{veiga-buffa}. 
\begin{figure}[!t]
  \centering
    \includegraphics[scale=0.8]{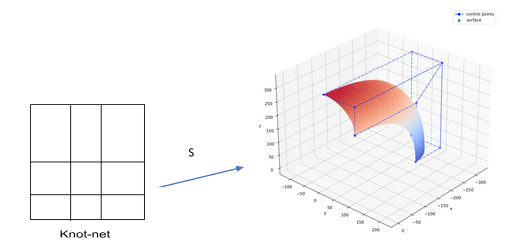}
    \caption{A NURBS surface patch with its control mesh and knot-net.}
    \label{fig:domains}
\end{figure}
\subsection{NURBS-Enhanced Finite Element Method}
Starting from this section, we proceed with a detailed description of our methodology.

\subsubsection{The Hybrid Mesh} \label{sec:hybrid_mesh}
Since we assume that $\partial \Omega$ is represented by a single NURBS patch $\mathcal{P}$, we can use a single NURBS parametrization to define $\partial \Omega$. 
We further assume that each element in $\mathcal{T}_h^{(b)}$ can at most have one face on the curved boundary, therefore, a boundary layer element can have at most five curved faces - one NURBS face and four interior curved faces each of which has an edge on $\mathcal{P}$. 
\subsubsection{Transformations}\label{sec:transforms}
\indent We construct the elements in $\mathcal{T}_h^{(i)}$ by isoparametric transformations of the reference cube $\hat{\mathcal{Q}}=[0,1]^3$ and use a rational geometric map to define 
the elements in $\mathcal{T}_h^{(b)}$. The geometric map used to construct the boundary layer elements involve the NURBS surface map that yield the boundary faces of the hexahedrals in $\Omega_{\mathcal{B}}$.

\noindent We denote by $\tilde{F}_\mathcal{Q}:\hat{\mathcal{Q}} \to \mathcal{Q}$ the local geometric transformation used in defining the boundary layer elements $\mathcal{Q}\in \mathcal{T}_h^b$.  
By scaling and translation, we define a linear transformation $T_\mathcal{Q}: [0,1]^2\to \hat{I}_\mathcal{Q}$, where $\hat{I}_\mathcal{Q}\in \mathcal{T}_S$ denotes the pre-image of $I_\mathcal{Q}\in \mathcal{T}_S^p$. Then, we define a local NURBS map $S_\mathcal{Q}:[0,1]^2\to I_\mathcal{Q}$ as the pull-back of $S$ by $T_\mathcal{Q}$, that is, $S_\mathcal{Q}:=S\circ T_\mathcal{Q}$. \\
As in \cite{intro-iga}, we assume that $S$ is a bi-Lipschitz homeomorphism such that  $S|_{\hat{I}_\mathcal{Q}^{ext}}\in C^{\infty}(\hat{I}_\mathcal{Q}^{ext})$ for all $\hat{I}_\mathcal{Q}\in \mathcal{T}_S$, where $\hat{I}_\mathcal{Q}^{ext}$ is the support extension of $\hat{I}_\mathcal{Q}$ and  $S^{-1}|_{I_\mathcal{Q}^{ext}}\in C^{\infty}(I_\mathcal{Q}^{ext})$ for all $I_\mathcal{Q}\in \mathcal{T}_S^p$, where $I_\mathcal{Q}^{ext}$ is the support extension of $I_\mathcal{Q}$.
This assumption eliminates the potential self-intersections and singularities such as those that commonly occur when a rectangular element is mapped to a curvilinear triangular element \cite{intro-iga}.\\
\noindent Let $\tilde{F}^{(1)}_\mathcal{Q}$ denote the candidate definitions of $\tilde{F}_\mathcal{Q}$ that we consider and use to derive $\tilde{F}_\mathcal{Q}$. For example, by utilizing the blending function method defined in \cite{blend}, we may define $\tilde{F}_\mathcal{Q}$ as:
\begin{align} \label{eq:blend-map}
    \tilde{F}^{(1)}_\mathcal{Q}(\hat{x})= \sum\limits_{i=1}^8 N_i(\hat{x}) X_i +f_1(\hat{x}), 
\end{align}
where $X_i$ denotes the global coordinates of the node corresponding to $\hat{x}_i=(\alpha_i,\beta_i,\zeta_i)\in [-1,1]^3$, and $f_1(\hat{x})$ denotes the face blending function defined as follows:
\begin{align*}
    f_1(\hat{x})&=\Big( S_\mathcal{Q}(\alpha,\beta)- \frac{1}{4}\big((1-\alpha)(1-\beta)X_{5}+(1+\alpha)(1-\beta)\\ X_{6}
    &+(1+\alpha)(1+\beta)X_{7}+(1-\alpha)(1+\beta)X_{8}\big)\Big)\big(\frac{1-\zeta}{2}).
\end{align*}
\noindent Here, $f_1(\hat{x})$ corresponds to the multiplication of the difference between the curved face and the planar face connecting the vertices of the face with a blending term\cite{blend}. Note that the first term in \eqref{eq:blend-map} corresponds to the standard isoparametric map for $\mathbb{Q}_1$ Lagrange elements in $\mathbb{R}^3$.
\begin{remark}
\noindent It is easy to show that 
\eqref{eq:blend-map} can be written in the following format. We refer to the blending function method for the generality and flexibility it offers. However, we will be utilizing the following format of \eqref{eq:blend-map} from \cite{corbett} as it offers an easier interpretation. 
\begin{align}\label{eq:c-map}
    \tilde{F}_{\mathcal{Q}}^{(1)}(\hat{x})= \sum\limits_{i=1}^{n_{cp}}& R_{i}(\alpha,\beta) C_{i} \frac{1}{2}(1-\zeta)+ \frac{1+\zeta}{8}\Big( (1-\alpha)(1-\beta)X_1\notag \\ &+(1-\alpha)(1+\beta)X_4+(1+\alpha)(1-\beta)X_2\notag \\ &+(1+\alpha)(1+\beta)X_3\Big),
\end{align}
\noindent where $\{R_{i}(\cdot,\cdot)\}$ are NURBS basis functions, $i:=(i_2-1)n_1+i_1$ with $n_1$ denoting the number of basis functions in the first parametric direction and  $\{i_k\}_{k=1}^2$ denoting the indices of the relevant basis function in 1D, and $n_{cp}$ is the number of control points $\{C_{i}\}$ associated with the NURBS face given by the map $S_\mathcal{Q}$.
\end{remark}
\noindent Since control points do not lie on the physical domain, their pre-images under the NURBS map are not well-defined in the reference domain. Therefore, we do not use the NURBS basis functions and the control points to define $\tilde{F}_{\mathcal{Q}}$ as in \eqref{eq:c-map}. Instead, we utilize the transformed NURBS basis functions and the transformed control points introduced in \cite{transformed-nurbs1,transformed-nurbs2} while also ensuring the reference domain for both NURBS and finite element basis functions is $[0,1]^d$ for an appropriate value of $d$ (unlike in \eqref{eq:blend-map}). Thus, we define the geometric map $\tilde{F}_{\mathcal{Q}}$ as follows:
\begin{align}\label{eq:blendedmap-transformed}
    \tilde{F}_{\mathcal{Q}}(\hat{x})= \sum\limits_{i=1}^{n_{cp}} \hat{R}_{i}(\alpha,\beta) \hat{C}_{i} (1-\zeta)+ &\zeta\Big( (1-\alpha)(1-\beta)X_1+\alpha(1-\beta)X_2& \\+\alpha \beta X_3+(1-\alpha)\beta X_4\Big), \notag
\end{align}
\noindent where $\hat{R}= T^{-1}R$ is the vector of transformed NURBS basis functions and $\hat{C}= T^t C$ is the vector of transformed control points obtained via the transformation matrix $T$ with entries  $T_{ij}:=[R_{i}(x_j)]$ where $x_{j}$ is the image of the Greville point $\gamma_{j}$ in \cite{transformed-nurbs1}. \\
\noindent Unlike the original control points, the transformed control points lie in the physical domain. In addition, the transformed NURBS basis functions and transformed control points preserve the exact geometry representation offered by the original NURBS basis functions and control points \cite{transformed-nurbs1,transformed-nurbs2}. 

\noindent Then, we define the local geometric transformation $F_{\mathcal{Q}}:\hat{\mathcal{Q}} \to \mathcal{Q}$ in a piecewise manner as follows:
\begin{align}\label{eq:geom-maps}
F_{\mathcal{Q}}(\hat{x}) := 
\begin{cases}
    \tilde{F}_{\mathcal{Q}}(\hat{x}), & \mathcal{Q} \in \mathcal{T}^{(b)}_h \\
    \bar{F}_{\mathcal{Q}}(\hat{x}), & \mathcal{Q} \in \mathcal{T}_h^{(i)}
\end{cases}
\end{align}
where $\bar{F}_{\mathcal{Q}}$ denotes the geometric map used with the $\mathbb{Q}_1$ Lagrange element (See Figure~\ref{fig:nefem_cubes}).  We assume that $det(D\bar{F}_{\mathcal{Q}}(\hat{x}))>0$ for $\forall \hat{x} \in \hat{Q}$. We then define the global geometric map $F$ such that $F|_{\mathcal{Q}}:=F_{\mathcal{Q}}$ for $\mathcal{Q}\in \mathcal{T}_h$.\\ 
We note that since the face blending term vanishes over the straight faces connecting the interior elements to the NURBS-enhanced boundary elements, the mapping is continuous over the interior face of the boundary layer and the mesh is geometrically conforming as the intersection of any boundary layer element with an interior element is a mesh face. 

\begin{figure}[H]
    \centering
    \includegraphics[scale=0.5]{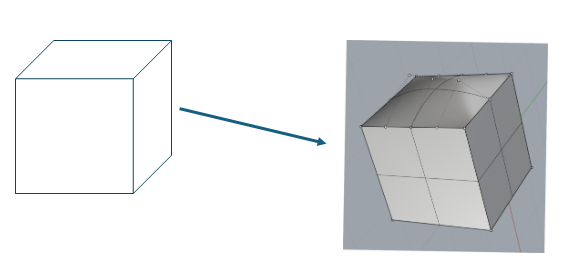}
    \caption{The reference cube (on the left), and four $\mathcal{Q}\in \mathcal{T}_h^b$ and $\mathcal{Q}\in \mathcal{T}_h^i$ (on the right).}
    \label{fig:nefem_cubes}
\end{figure}

\begin{remark}
Based on the curvature of the domain, we may choose to use fully NURBS-enhanced elements (that is, elements with six NURBS faces as the one illustrated in Figure~\ref{fig:my_hex}) in the boundary layer. The analysis in such a layer would be identical to isogeometric analysis, therefore, we would identify the elements in this layer as isogeometric elements and redefine $\Omega_\mathcal{B}$ as \enquote{the transition layer} between the isogeometric elements and the finite elements rather than \enquote{the boundary layer}.
\end{remark}
\begin{figure}[H]
    \centering
       \includegraphics[scale=0.23]{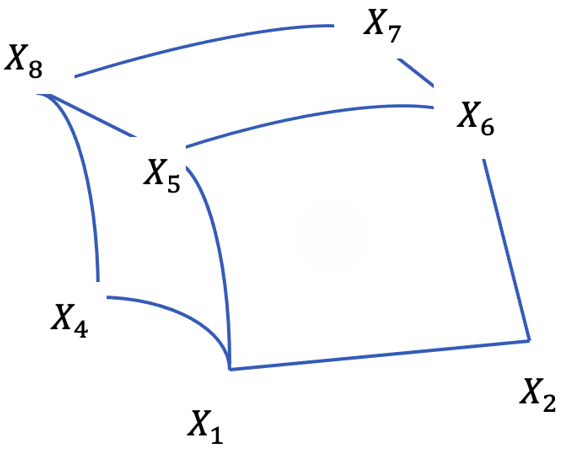}
    \caption{Fully NURBS-Enhanced element with the reference cube}
    \label{fig:my_hex}
\end{figure}
\subsection{Function Spaces}
In this section, we define the function spaces by utilizing the following definition of a \textit{finite element space} by Ciarlett \cite{ciarlet1972}.
\begin{definition}
A finite element space is a triple $(\mathcal{Q}, \mathbb{P}_\mathcal{Q}, \Sigma_{\mathcal{Q}})$ that satisfies \cite{ciarlett}:
\begin{itemize}
    \item $\mathcal{Q}\subset \mathbb{R}^n$ is a closed subset with a non-empty interior and a Lipschitz-continuous boundary, (element domain)
    \item $\mathbb{P}_\mathcal{Q}$ is a finite dimensional space of real-valued functions over $\mathcal{Q}$, (finite element space).
    \item $\Sigma_{\mathcal{Q}}$ is a finite set of linearly-independent linear forms over $C^{\infty}(\mathcal{Q})$, (degrees of freedom, a.k.a, nodes).
\end{itemize}
\noindent In the finite element theory, $\Sigma_\mathcal{Q}$ is assumed to be $\mathbb{P}_\mathcal{Q}$-unisolvent, that is, any function $p\in \mathbb{P}_\mathcal{Q}$ that nullifies $\Sigma_\mathcal{Q}$ is identically zero. Thus, any $p\in \mathbb{P}_\mathcal{Q}$ can be written as $p=\sum\limits_{i=1}^n \psi_i(p) p_i$ where $\psi_i\in \Sigma_\mathcal{Q}$ and $\{p_i\}_{i=1}^N$ are basis functions (\enquote{generally} polynomial functions) that span $\mathbb{P}_{\mathcal{Q}}$. 
\end{definition}
\noindent In alignment with Ciarlett's definition of a finite element space above, we first define the finite element spaces locally. 
Then, using these local finite element spaces, we define the global finite element spaces. Boundary conditions can be imposed on the global spaces based on the problem of interest.\\
\indent Since we construct a hybrid global finite element space that consists of $\mathbb{Q}_1$ Lagrange finite element spaces used for obtaining approximations in $\Omega_{int}$ and NURBS-enhanced finite element spaces used for obtaining approximations in $\Omega_{\mathcal{B}}$, we define two generic local finite elements. However, it suffices to verify the unisolvency condition for the NURBS-enhanced finite element spaces only. We note that our finite element construction yields interelement continuity on the global scale despite the disparity in the definition of finite element spaces at the local level.
\subsubsection{Basis functions}
\noindent We use the standard nodal basis functions for the piecewise linear finite element space and define the following set of basis functions for the NURBS-enhanced finite element space:
\begin{align}\label{eq:basis-hybrid-R-transformed}
    &\tilde{N}_i= \hat{R}_i(\alpha,\beta)(1-\zeta),\quad \quad \text{for}\ i=1,\dots, n_{cp},\\
    &\tilde{N}_{n_{cp}+1}=(1-\alpha)(1-\beta)\zeta, \notag\\
    &\tilde{N}_{n_{cp}+2}=\alpha(1-\beta)\zeta,\notag\\
    &\tilde{N}_{n_{cp}+3}=\alpha \beta \zeta,\notag \\
    &\tilde{N}_{n_{cp}+4}=(1-\alpha)\beta \zeta.\notag
\end{align}
\noindent Here, $\{\hat{R}_i\}$ are the transformed NURBS basis functions mentioned in Section~\ref{sec:meshes}, $i=n_1(i_2-1)+i_1$ denotes the index of a basis function in 2D, where $n_1$ denotes the number of basis functions in the first parametric direction and $\{i_k\}_{k=1}^2$ denote the indices of the relevant basis function in 1D as mentioned earlier.\\
\noindent The transformed basis functions are related
to Greville points in the parametric domain and the images of Greville points in the physical domain. They are interpolatory at the images of the Greville points, that is, $\hat{R}_{k}(x_m)=\delta_{km}$ where $x_{m}= S(\gamma_{m})$, and they form a partition of unity as shown in Lemma~\ref{lem:transformed-basis_p_unity}.
\begin{lemma}\label{lem:transformed-basis_p_unity}
    \begin{align*}
    \sum\limits_{k=1}^{n_{cp}} \hat{R}_k(x)=1
\end{align*}
\end{lemma}
\begin{proof}
Using the definition of the transformed NURBS basis functions stated in \cite{transformed-nurbs1} and the partition of unity property of the NURBS basis functions, we obtain
     \begin{align*}
    \sum\limits_{k=1}^{n_{cp}} R_k(x)&=\sum\limits_{k=1}^{n_{cp}}\sum\limits_{m=1}^{n_{cp}}T_{km}\hat{R}_m(x)=\sum\limits_{m=1}^{n_{cp}}\sum\limits_{k=1}^{n_{cp}}T_{km}\hat{R}_m(x)
=\sum\limits_{m=1}^{n_{cp}}\sum\limits_{k=1}^{n_{cp}}R_k(x_m)\hat{R}_m(x),\\1&=\sum\limits_{m=1}^{n_{cp}}\Big(\sum\limits_{k=1}^{n_{cp}}R_k(x_m)\Big) \hat{R}_m(x)=\sum\limits_{m=1}^{n_{cp}}\hat{R}_m(x)
\end{align*}
\end{proof}
\noindent Moreover, the transformed NURBS basis functions span the same solution spaces as the original NURBS basis functions due to the linearity of the basis transformation\cite{transformed-nurbs1}. Although they do not strictly have local compact support, they were shown to rapidly decay from their unit peak values in \cite{transformed-nurbs1}. 
A NURBS surface described by isogeometric basis functions $\{R_i\}$ and the corresponding control points $\{C_i\}$ via a NURBS map $S$ can be described exactly by the transformed NURBS basis functions and the transformed control points, that is,
\begin{equation}
    S=\sum\limits_{i=1}^{n_{cp}} R_iC_i=\sum\limits_{i=1}^{n_{cp}} \hat{R}_i \hat{C}_i,
\end{equation}
where $\hat{C}_i:=\sum\limits_{j=1}^{n_{cp}} T_{ji}C_i$. 
Under this basis transformation, the isogeometric approximation of a scalar physical field at the image of a Greville point $\gamma_{k}$ denoted by $x_{k}$ 
is given by\cite{transformed-nurbs1}
\begin{align*}
    \hat{d}_{k}= u_h(x_{k})=\sum\limits_{i=1}^{n_{cp}} T_{ik}d_{i},\quad \text{with}\ T_{{i}{k}}=R_{i}(x_{k}),
\end{align*}
\noindent where $\{d_{\textbf{i}}\}_{i=1}^{n_{cp}}$ denote the coefficients of the isogeometric approximation corresponding to the original control points.\\ 
We use the transformed basis functions in determining the values of physical fields over the entire boundary via the Greville points,  unlike in \cite{transformed-nurbs1,transformed-nurbs2}.
\begin{figure}[H] 
\centering
\begin{minipage}[b]{0.4\textwidth}
\includegraphics[width=\textwidth]{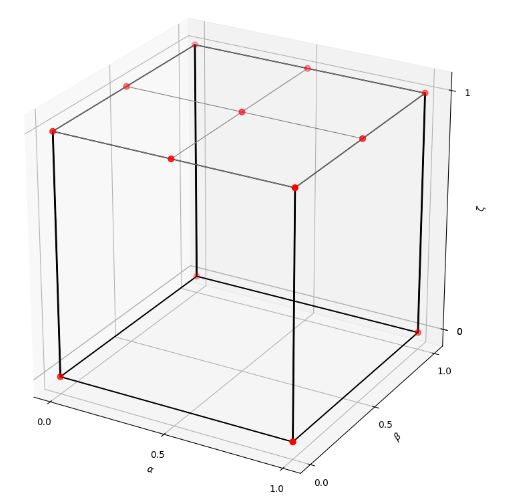}
\caption*{(a)}
\end{minipage}
\begin{minipage}[b]{0.5\textwidth}
\includegraphics[width=\textwidth]{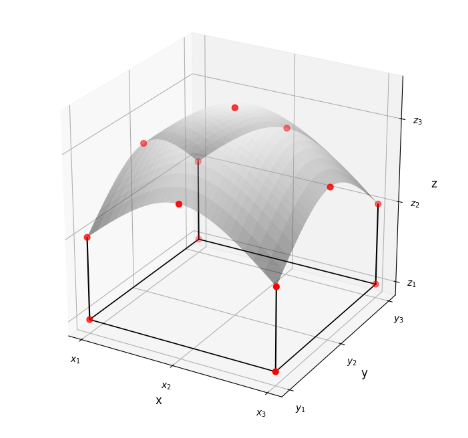}
\caption*{(b)}
\end{minipage}
\caption{Assume $p_1=p_2=2$. Degrees of freedom for $\mathcal{Q}\in\mathcal{T}_h^{b}$ are illustrated (in red) in the reference domain in (a) and in the physical domain in (b).}
\label{fig:dofs}
\end{figure}
\subsubsection{Local Finite Element Spaces}
\noindent Let $\mathbb{Q}_1(\hat{\mathcal{Q}})$ denote the piecewise-linear Lagrange finite element space defined over the reference element $\hat{\mathcal{Q}}$. Then, we define
the local finite element space over an arbitrary element $\mathcal{Q}$ in the physical subdomain $\Omega_{int}$ as
\begin{align*}
V_h^{int}(\mathcal{Q})&= \{ v : v= \hat{v}\circ \bar{F}_\mathcal{Q}^{-1}, \hat{v}\in \mathbb{Q}_1(\hat{\mathcal{Q}}) \}, \end{align*}
with the standard nodal degrees of freedom:
\begin{equation*}
S_{V_h^{int}}= \{v(a_i): a_i\in \mathcal{Q},\ 1\le i\le 8\}.
\end{equation*}

\noindent Let $V_h^b(\hat{\mathcal{Q}})$ denote the NURBS-enhanced finite element space defined over the reference element $\hat{\mathcal{Q}}$ using the basis functions listed in \eqref{eq:basis-hybrid-R-transformed}. 
\noindent Thus,  
\begin{align}
    V_h^b(\hat{\mathcal{Q}}):=span\{\tilde{N}_i: 1\le i\le n_{cp}+4\}
\end{align}
with the degrees of freedom 
\begin{align*}
S_{\hat{V}_{h}^b}&=\{\hat{v}(\hat{a}_{\textbf{i}}): \hat{a}_{\textbf{i}}\in \hat{\mathcal{Q}},\  \text{where}\
\hat{a}_i=(\hat{a}_{\textbf{i}}^1,\hat{a}_{\textbf{i}}^2, 1)\ 
\text{for}\  (n_{cp}+1)\le i\le (n_{cp} +4),\  \hat{a}_i=(\gamma_{\textbf{i}}^1,\gamma_{\textbf{i}}^2,0)\\ &
\text{for}\  1\le i\le n_{cp}\ \text{and}\
\gamma_{\textbf{i}}\in \mathcal{G}_h\},
\end{align*}
\noindent where the tuples $\{(\hat{a}_{\textbf{i}}^1,\hat{a}_{\textbf{i}}^2)\}$ denote the coordinates of the $\mathbb{Q}_1$ finite element nodes over $[0,1]^2$  and
$\mathcal{G}_h$ indicates the Greville mesh associated with the face of $\hat{\mathcal{Q}}$ that serves as the reference domain of the NURBS face of $\mathcal{Q}\in \mathcal{T}_h^b$ Figure~\ref{fig:dofs} illustrates the degrees of freedom in the reference and physical domains. For illustrative purposes, the parametric domain of the NURBS face is shown at $\zeta=1$ (instead of $\zeta=0$ as in the formulations).
\begin{lemma}\label{lem:unisolvency}
$S_{\hat{V}_h^b}$ unisolves $V_h^b(\hat{\mathcal{Q}})$. 
\end{lemma}

\begin{proof}
First, note that the dimension of $V_h^b(\hat{\mathcal{Q}})$ equals the cardinality of $S_{\hat{V}_h^b}$. Then, let $\hat{v}\in V_h^b(\hat{\mathcal{Q}})$, thus, write $\hat{v}(\hat{x})=\sum\limits_{i=1}^{n_{cp}+4} l_i \tilde{N}_i(\hat{x})$. Now, suppose that $\hat{v}(\hat{a}_j)=0$ for $\forall j\in \{1,\dots, n_{cp}+4\}$.
Since $\tilde{N}_i(a_j)=\delta_{ij}$, this yields
\begin{align*}
    \hat{v}(\hat{a}_j)&= \sum\limits_{i=1}^{n_{cp}+4} l_i \tilde{N}_i(\hat{a}_j)=l_j=0, \quad \forall j\in \{1,2,\dots,n_{cp}+4\}.
\end{align*}
Thus, $\hat{v}=0$ on $\hat{\mathcal{Q}}$.
\end{proof}

\noindent Then, we define the local NURBS-enhanced finite element space over an arbitray element $\mathcal{Q}$ in the physical subdomain $\Omega_{\mathcal{B}}$ as follows:
\begin{align*}
V_h^{b}(\mathcal{Q})&= \{ v: v= \hat{v}\circ \tilde{F}_\mathcal{Q}^{-1}, \hat{v}\in V_h^{b}(\hat{\mathcal{Q}}) \}. 
\end{align*}
\noindent Let $S_{V_{h}^b}$ be the set of degrees of freedom that correspond to the degrees of freedom listed in $S_{\hat{V}_{h}^b}$ in the physical domain. 
The invertibility of $\tilde{F}_\mathcal{Q}$ and the unisolvency of $S_{\hat{V}_{h}^b}$ implies that $S_{V_{h}^b}$ unisolves $V_h^{b}(\mathcal{Q})$.
\begin{remark}
    \noindent Since $\tilde{F}_\mathcal{Q}$ is not arbitrarily regular, a function in $H^s(\hat{\mathcal{\mathcal{Q}}})$ may not be mapped to a function in $H^s(\mathcal{\mathcal{Q}})$ under $\tilde{F}_\mathcal{Q}$, that is, having $V_h^b(\hat{\mathcal{Q}})\subset H^1(\hat{\mathcal{Q}})$ would not necessarily yield an $H^1$ space in the physical domain. 
    However, since $\tilde{F}_\mathcal{Q}$ is sufficiently regular due to our assumptions on $S$ and $S^{-1}$ in Section~\ref{sec:transforms}, we can assume that $V_h^b(\mathcal{Q})\subset H^1(\mathcal{Q})$. Furthermore, since we assume that the splines are at least continuous, below we can define $V_h^b$ such that $V_h^b\subset H^1(\Omega_{\mathcal{B}})$.
\end{remark}
\subsubsection{Global Finite Element Space}
We define the global finite element space for $\Omega_{int}$ as follows: \begin{align*}
V_h^{int}&= \{ v \in H^1(\Omega_{int}): v|_{\mathcal{Q}}\in V_h^{int}(\mathcal{Q}),\ v\ \text{is continuous at}\\ & \text{all vertices}\ a_i\ \text{in}\ \mathcal{T}_h^{(i)},\ \forall \mathcal{Q}\in \mathcal{T}_h^{(i)} \}, \end{align*}
\noindent Then, we define the global finite element space for $\Omega_{\mathcal{B}}$ as: \begin{align*}\label{eq:global-fespace}
V_h^{b}&= \{ v \in H^1(\Omega_{\mathcal{B}}): v|_{\mathcal{Q}}\in V_h^b(\mathcal{Q}),\ v\ \text{is continuous at}\\ & \text{the interior nodes}\ a_i\ \text{in}\ \mathcal{T}_h^{(b)},\ \forall \mathcal{Q}\in \mathcal{T}_h^{(b)}\}, 
\end{align*}
Thus, the global function space over $\Omega$ becomes:
\begin{equation}\label{eq:global-fespace}
    V_h:=\{v\in H^1(\Omega): v|_{\Omega_{int}}\in V_h^{int}, v|_{\Omega_{\mathcal{B}}}\in V_h^b \}
\end{equation}
\noindent And we may write $V_h:= V_h^{int}\oplus V_h^b$.
\color{black}
\subsubsection{Global Finite Element Spaces with Boundary Conditions}
To impose homogeneous Dirichlet boundary conditions on the global finite element spaces, it suffices to modify the global space defined over the boundary layer of the domain, thus, we first define
\begin{align*}
V_h^{b,0}(\mathcal{Q})&= \{ v: v= \hat{v}\circ \tilde{F}_\mathcal{Q}^{-1}, \hat{v}\in V_h^{b}(\hat{\mathcal{Q}}), v|_{\mathcal{Q}\cap \partial \Omega}=0 \}. 
\end{align*}
\noindent Then, the global finite element space for $\Omega_{\mathcal{B}}$ becomes: 
\begin{align*}
V_h^{b,0}&= \{ v \in H^1(\Omega_{\mathcal{B}}): v|_{\mathcal{Q}}\in V_h^{b,0}(\mathcal{Q}),\ v\ \text{is continuous at}\\ & \text{at the interior nodes}\ a_i\ \text{in}\ \mathcal{T}_h^{(b)},\ \forall \mathcal{Q}\in \mathcal{T}_h^{(b)}\}.
\end{align*}
\noindent This yields the following global function space over $\Omega$:
\begin{align}\label{eq:global-fespace-bc}
    V_h^0:=\{v\in H^1_0(\Omega): v|_{\Omega_{int}}\in V_h^{int}, v|_{\Omega_{\mathcal{B}}}\in V_h^{b,0} \}
\end{align}
As in earlier, we may write $V_{h,0}:= V_h^{int}\oplus V_h^{b,0}$.\\
Note that in the Galerkin formulation of isogeometric analysis, homogeneous boundary conditions can be exactly enforced by setting the control variables to zero \cite{hughes}.
\subsection{Numerical Integration}
Classical Gaussian quadrature rules consisting of $n$ points allow an exact integration of polynomials of order $(2n-1)$
\cite{gen_gauss}. 
In Galerkin-type formulations, element-wise Gaussian quadrature is optimal for standard finite element methods and has been used extensively for quadrilateral and hexahedral finite elements\cite{hughes_ni}. 
However, it has been shown that Gaussian quadrature rule is suboptimal in isogeometric analysis since it disregards the interelement continuity of the smooth spline basis functions \cite{hughes_ni,zou21,zou22}. To construct efficient quadrature rules for isogeometric analysis,  the interelement continuity levels of splines need to be taken into account. Thus, rules that involve more than a single element need to be constructed as smoothness across the element boundaries would reduce the number of degrees of freedom that would be required by the standard continuous finite elements for the same mesh. Such reductions in the number of degrees of freedom (therefore, the number of basis functions) yield reductions also in the number of quadrature points required for exact integration. Various efficient quadrature rules with reduced sets of quadrature points were presented in \cite{auricc,zou22, hughes_ni}. \\
\noindent Zou et al. \cite{zou21} have shown that Greville quadrature rules yield comparable accuracy as full Gaussian quadrature rules but are significantly more efficient than the full Gaussian quadrature in isogeometric analysis. The authors, however, also pointed out that Greville quadrature weights may be negative in regions where element sizes change abruptly. Therefore, non-uniform knot vectors may involve negative Greville quadrature weights. Negative quadrature weights may cause instability, thus, are not preferred in numerical or engineering analysis. A Gauss-Greville quadrature rule was introduced in \cite{zou22} to address these negative quadrature weight issues where they occur by means of various checks and adaptations of the quadrature rule to be employed. \\
\indent In this manuscript, our goal is to provide quadrature rules that would yield optimal accuracy for the numerical approximations of the integrals involved in the variational formulations of second-order partial differential equations. In this regard, we employ the Gauss-Legendre quadrature rule\cite{quarteroni} for the hexahedral elements in $\Omega_{int}$, and for the elements in $\Omega_{\mathcal{B}}$, we propose a new quadrature rule which blends the Greville quadrature rule introduced in \cite{zou21} and the Gauss-Legendre quadrature rule. We assume that the knot vectors forming the NURBS surfaces are uniform, thus, the Greville weights are non-negative. Multi-dimensional quadrature rules are obtained from the tensor product of suitable one-dimensional rules in both isogeometric analysis and the finite element analysis carried over tensor-product meshes.  Since the spline basis functions and the finite element basis functions we utilize are tensor products of piecewise-polynomials in one-dimension, we use one-dimensional quadrature rules to derive the two-dimensional quadrature rules that we need to utilize in the derivation of our blended quadrature rule in three-dimensions.\\ 
Recall that a 1D function $f$ can be numerically integrated over a domain $I$ by a set of tuples $\{(x_i,w_i)\}_{i=1}^n$ that satisfy:
\begin{align}\label{eq:gen-quad}
    \int\limits_{I} f(\alpha) d\alpha\approx\sum\limits_{i=1}^n f(x_i)w_i
\end{align}
where $\{x_i\}$'s are called the quadrature points and $\{w_i\}$'s are the weights associated with them. Note that the quadrature order must be sufficiently high for the integrand so that \eqref{eq:gen-quad} can be written as an equality rather than an approximation.\\
\indent The Greville  quadrature rule for B-splines uses the Greville points associated with the B-spline basis functions as the quadrature points and determines the quadrature weights by solving a linear moment fitting problem in each parametric direction\cite{zou22}. Precisely, if we let $\{x_i\}$ be the Greville points calculated as in Equation~\eqref{eq:grevpts}, then the weights $\{w_i\}$ are determined by solving the moment-fitting system of equations of the form\cite{zou22}:
\begin{align}\label{eq:quad_w}
\begin{bmatrix}
\int\limits_{I} N_1(\eta)d\eta\\
\int\limits_{I} N_2(\eta)d\eta\\
.\\
.\\
.\\
\int\limits_{I} N_n(\eta)d\eta
\end{bmatrix} =\begin{bmatrix}
 N_1(x_1)\ N_1(x_2) \dots N_1(x_n) \\
 N_2(x_1)\ N_2(x_2) \dots N_2(x_n) \\
 \dots \\
 \dots \\
 \dots\\
 N_n(x_1)\ N_n(x_2) \dots N_n(x_n) \\
\end{bmatrix}\begin{bmatrix}
w_1\\
w_2\\
.\\
.\\
.\\
w_n
\end{bmatrix}
\end{align}
where $N_i$ denotes a B-spline basis function of degree $p$ and the moments on the left-hand side can be calculated as shown in \cite{johannessen, chui2004nonstationary}.
\noindent  Therefore, the quadrature rule can exactly integrate all linear combinations of the univariate B-spline basis functions $\{N_i\}_{i=1}^n$ if the determinant of the matrix in Equation~\eqref{eq:quad_w} is non-zero \cite{zou21}.\\
\indent To ensure numerical stability, discrete Galerkin forms must be rank-sufficient under quadrature. In \cite{zou21}, it is mentioned that using the Greville points as the quadrature points yields stiffness and mass matrices that are free of rank deficiency regardless of the mesh sizes and the polynomial degrees. However, the authors also noted that even though the matrices are full rank, the system can still suffer from spurious modes due to a lack of integration accuracy. Thus, 
accurate integration of the spline basis functions is still necessary for eliminating spurious modes and obtaining a well-conditioned system and accurate results\cite{zou21}. This issue was addressed by defining a reduced Gaussian quadrature rule which yields more accurate integration than the Greville quadrature for polynomial degrees $p>2$ even with less quadrature points than the full Gaussian quadrature in \cite{zou22}, but it was also noted that for $p=2$, both the Greville quadrature and the reduced Gaussian quadrature rules yield optimal convergence. In this research, it suffices to use B-spline basis functions with $p=2$, thus, with the Greville quadrature rule, we can still obtain optimal convergence while numerically integrating functions over the surface.\\ 
\indent Note that while working with the variational formulations involving NURBS basis, one would encounter integrals of the following form
\begin{align}\label{eq:sampleint}
\int R_i(\alpha) R_j(\beta)  J_c(\alpha, \beta)
\end{align}
\noindent where $J_c(\alpha,\beta)$ is a function that involves the Jacobian (or the inverse of the Jacobian) of the geometric map and the coefficients of the partial differential equation. Since both the NURBS map and the weight function $W(\alpha, \beta)$ are piecewise smooth functions defined on the initial coarse mesh where the geometry is exactly represented and they remain unchanged during refinement, it can be assumed that $J_c(\alpha,\beta)$ and $W(\alpha,\beta)$ are constant as in \cite{hughes_ni}.  Thus, while determining the quadrature rule  for computing the integrals that involve NURBS basis functions in our formulations, we can utilize the Greville quadrature rule that can be used for evaluating integrals involving B-spline basis functions. In other words, it suffices to utilize a quadrature rule that would exactly approximate the following integral instead of the integral in~\eqref{eq:sampleint}:
\begin{align}\label{eq:sampleint_red}
\int N_i(\alpha)N_j(\beta)
\end{align}
\noindent Therefore, for the NURBS-enhanced elements, we construct a new quadrature rule by blending the Greville quadrature rule for the B-splines and the Gauss-Legendre quadrature rule in two dimensions. 
First, by scaling and translation, we redefine the domain of the geometric map as $[-1,1]^3$ since the Gauss-Legendre quadrature rule in 3D is defined over $[-1,1]^3$. (It is easy to see that $G: \mathbb{R}^3 \to \mathbb{R}^3$ such that $G(v)=2v-\vec{1}$ would transform $[0,1]^3$ into $[-1,1]^3$). 
Then, the initial procedure for constructing the new quadrature rule is as follows: 
\begin{itemize}
    \item Let $\tilde{q}=(0,0)$ be the single Gauss-Legendre quadrature point with weight $w=4$ over $[-1,1]^2$. Then, move it to the $z=1$ plane and obtain a quadrature point $q^G=(0,0,1)$, keeping $w$ unchanged.
    \item Suppose $\{\gamma_{ij}:=(\gamma_i,\gamma_j)\}\in [-1,1]^2$ are the pulled-back Greville quadrature points with weights $\{w_{ij}:=w_iw_j\}$. Then, move them to $z=-1$ plane to get the quadrature points of the form $q^{Gr}_{ij}=(\gamma_i, \gamma_j,-1)$. The weights associated with $w_{i}$ (similarly, $w_j$) can be computed by solving the system in Equation~\eqref{eq:quad_w} using the one-dimensional set of B-spline basis functions involved in the description of the NURBS-enhanced element. 
\end{itemize}
Now, we can derive the quadrature points for the NURBS-enhanced elements by weighting and scaling as follows:
\begin{align}
    \tilde{q}^{\mathcal{B}}_{k}&:=(\frac{q^G_x\sqrt{w}+q^{Gr}_xw_{i}}{\sqrt{w}+w_i}, \frac{q^G_y\sqrt{w}+q^{Gr}_yw_{j}}{{\sqrt{w}+w_j}},\frac{q^G_zw+q^{Gr}_zw_{ij}}{w+w_{ij}}),\ k=1,2,\dots,n_{cp}, \label{eq:hybrid_quad_pts-new}
\end{align}
where  $k:=(j-1)n_1+i$ and we denoted the $x$ ($y$ and $z$, resp.) component of $q^{Gr}_{ij}$ by $q^{Gr}_x$ ($q^{Gr}_y$, $q^{Gr}_z$, resp.) dropping the point indices $i$ and $j$ to simplify notation.\\
Note that the novel quadrature points also lie in $[-1,1]^3$ and they carry over the influence of the weights associated with the quadrature points originating from the opposite faces of the cube $[-1,1]^3$ in a consistently proportional and systematic manner. Finally, the weights corresponding to these quadrature points  can be computed via the general quadrature formulation~\eqref{eq:gen-quad} using the hybrid basis functions~\eqref{eq:basis-hybrid-R-transformed} as described below.\\ 
\indent Let $\hat{v}\in V_h^b(\hat{Q})$, then we may write $\hat{v}(\hat{x})=\sum\limits_{k=1}^{n_{cp}+4} \hat{v}(\hat{x}_k) \tilde{N}_k(\hat{x})$, where $\{\hat{x}_k\}$ denotes the set of nodes. This yields
\begin{align*}
    \int_{\hat{Q}} \hat{v}(\hat{x})\ d\hat{x}&=\sum\limits_{k=1}^{n_{cp}+4}\hat{v}(\hat{x}_k)\int_{\hat{Q}} \tilde{N}_k(\hat{x}) d\hat{x}
\end{align*}
Thus, to obtain a quadrature rule that would provide an exact integration of a function in $V_h^b(\hat{Q})$ over $\hat{Q}$, we need to ensure the exact integration of each $\tilde{N}_k$ over $\hat{Q}$, that is,  we need to find $\{w^{\mathcal{B}}_i\}_{i=1}^n$ such that

\begin{align}\label{eq:hybrid-quad-derive}
    \int_{\hat{Q}} \tilde{N}_k(\hat{x}) d\hat{x}=\int_{\hat{Q}_G} (\tilde{N}_k\circ G^{-1})(\tilde{x}) d\tilde{x}=
    \sum\limits_{i=1}^n \tilde{N}_k(q^{\mathcal{B}}_{i}) w^{\mathcal{B}}_i,\quad \text{for all}\ k=1,2,\dots,n_{cp+4},
\end{align}
where $\hat{Q}_G:=[-1,1]^3$, and the number of quadrature points $n:=n_{cp}$ by definition in \eqref{eq:hybrid_quad_pts-new}. We let $\hat{N}_k:=(\tilde{N}_k\circ G^{-1})$ and $m:=n_{cp}+4$ for notational simplicity, and write the overdetermined system that results from \eqref{eq:hybrid-quad-derive} in matrix form as follows:
\begin{align}\label{eq:quad_w_hybrid_overdetermined}
\begin{bmatrix}
\int\limits_{\hat{Q}_G} \hat{N}_1(\eta)d\eta\\
\int\limits_{\hat{Q}_G} \hat{N}_2(\eta)d\eta\\
.\\
.\\
.\\
\int\limits_{\hat{Q}_G} \hat{N}_{m}(\eta)d\eta
\end{bmatrix} =\begin{bmatrix}
 \hat{N}_1(q_1^{\mathcal{B}})\ \hat{N}_1(q_2^{\mathcal{B}}) \dots \hat{N}_1(q_n^{\mathcal{B}}) \\
\hat{N}_2(q_1^{\mathcal{B}})\ \hat{N}_2(q_2^{\mathcal{B}}) \dots \hat{N}_2(q_n^{\mathcal{B}}) \\
 \dots \\
 \dots \\
 \dots\\
 \hat{N}_{m}(q_1^{\mathcal{B}})\ \hat{N}_{m}(q_2^{\mathcal{B}}) \dots \hat{N}_{m}(q_n^{\mathcal{B}}) \\
\end{bmatrix}\begin{bmatrix}
w^{\mathcal{B}}_1\\
w^{\mathcal{B}}_2\\
.\\
.\\
.\\
w^{\mathcal{B}}_n
\end{bmatrix}
\end{align}
where we need to solve for the weight vector $\textbf{w}^{\mathcal{B}}:=[w_i^{\mathcal{B}}]$. Therefore, we define 
\begin{align}\label{eq:quad_w_hybrid_least-squares_p}
\mathbf{A}:=\begin{bmatrix}
 \hat{N}_1(q_1^{\mathcal{B}})\ \hat{N}_1(q_2^{\mathcal{B}}) \dots \hat{N}_1(q_n^{\mathcal{B}}) \\
\hat{N}_2(q_1^{\mathcal{B}})\ \hat{N}_2(q_2^{\mathcal{B}}) \dots \hat{N}_2(q_n^{\mathcal{B}}) \\
 \dots \\
 \dots \\
 \dots\\
 \hat{N}_{m}(q_1^{\mathcal{B}})\ \hat{N}_{m}(q_2^{\mathcal{B}}) \dots \hat{N}_{m}(q_n^{\mathcal{B}}) \\
\end{bmatrix},\quad 
\textbf{b}:=\begin{bmatrix}
\int\limits_{\hat{Q}_G} \hat{N}_1(\eta)d\eta\\
\int\limits_{\hat{Q}_G} \hat{N}_2(\eta)d\eta\\
.\\
.\\
.\\
\int\limits_{\hat{Q}_G} \hat{N}_{m}(\eta)d\eta
\end{bmatrix} 
\end{align}
As the system $\mathbf{A}\mathbf{w}^{\mathcal{B}}= \mathbf{b}$ does not yield a unique solution $\mathbf{w}^{\mathcal{B}}$, we use the least-squares method to find the $\mathbf{w}^{\mathcal{B}}$ that minimizes $\|\mathbf{A}\mathbf{w}^{\mathcal{B}}-\mathbf{b}\|^2$. In our case,  $n_{cp}\ge 4$ since there are $(p+1)^2$ basis functions (thus, control points and Greville points) active over each boundary element. 
\begin{remark}
    Note that some of the integrals would involve the differentials of the basis functions. For the sake of generality, we only focus on the integrals of the form \eqref{eq:sampleint_red}.  However, we note that the integration accuracy can be improved for the integrals involving the gradients of the basis functions by employing a modified version of our quadrature rule particularly because our basis functions are not all polynomials.
\end{remark}
\begin{remark}\label{rem:blended_quad_4}
    The derivation of a blended quadrature rule using a $4$-point Gauss-Legendre quadrature rule rather than a $1$-point Gauss-Legendre quadrature rule can be done in a similar fashion. Such derivation would require repeating~\eqref{eq:hybrid_quad_pts-new}  for each one of the Gaussian quadrature points over the face and yield $4n$ quadrature points.
    \end{remark}
\noindent Now we show how we would numerically compute volumetric and surface integrals in the physical domain using the quadrature rules mentioned above. Let $\hat{F}:=F\circ G^{-1}$. Consider
\begin{align}\label{eq:gen_int}
    \int_{Q} f(x)\ dx &=\int_{\hat{Q}_G} f(\hat{F}(\tilde{x}))\ |J_{\hat{F}}(\tilde{x})|\ d\tilde{x}
\end{align}
\noindent where $f\in L^2(\Omega)$ is an arbitrary function.\\ \\
\eqref{eq:gen_int} can be computed numerically as follows:
\begin{align*}
    \int_{\hat{Q}_G} f(\hat{F}(\tilde{x}))\ |J_{\hat{F}}(\tilde{x})|\ d\hat{x}
    &\approx \sum\limits_{l} \hat{f}(q^{(l)}) w^{(l)} |J_{\hat{F}}(q^{(l)})|
\end{align*}
\noindent where $\hat{f}:=f\circ \hat{F}$, $(q^{(l)}, w^{(l)})$ denotes a Gauss-Legendre quadrature point and weight pair if $Q\in \mathcal{T}_h^{(i)}$ and a blended quadrature point with its corresponding weight if $Q\in \mathcal{T}_h^{(b)}$.\\

\noindent Similarly, let $\hat{S}:= S\circ g^{-1}$, where $g$ is a bijective map from $[0,1]^2$ to $[-1,1]^2$ obtained via scaling and translation, and $Q_g:=[-1,1]^2$. Then, any scalar surface integral over a NURBS boundary face is given by

\begin{equation}
    \int_{\mathcal{P}} f(s)\ ds =\int_{\hat{Q}_g} f(\hat{S}(\tilde{s}))\ |J_{\hat{S}}(\tilde{s})|\ d\tilde{s}
\end{equation}
\noindent and numerically, such an integral can be approximated as follows
\begin{align*}
 \int_{\hat{Q}_g} f(\hat{S}(\tilde{s}))\ |J_{\hat{S}}(\tilde{s})|\ d\tilde{s}
    &\approx \sum\limits_{m} \hat{f}(q^{(m)}) w^{(m)} |J_{\hat{S}}(q^{(m)})|
\end{align*}
\noindent where $\hat{f}:=f\circ \hat{S}$, and $\{(q^{(m)},w^{(m)})\}$ are the images of two-dimensional Greville quadrature points under $g$ and their associated weights.

\subsection{Refinement}
\noindent In this subsection, we describe a new refinement strategy that can be employed for refining the NURBS-enhanced elements to complement our theoretical framework. In isogeometric analysis \cite{hughes}, $\Omega$ is often described via a coarse mesh consisting of only a few elements, while the approximate solution is computed on a refined mesh. Thus, the NURBS parametrization, $S$, which is defined over a patch and the weight function, $W$, are determined over the coarsest mesh.  During the refinement of the mesh and space, control points are adjusted to keep $S$ unchanged and the weights associated with the control points are adjusted to keep $W$ unchanged. Therefore, the refinement of splines is generally based on the refinement of function spaces while the underlying knot-spans are implicitly refined\cite{bazilevs2010}.\\ 
\noindent In this paper, we use $h$-refinement for finite elements and its equivalent in IGA, namely, the knot insertion for NURBS. As described in \cite{hughes}, after a knot insertion, the number of basis functions and the number of control points are increased by one although the geometry and the parametrization of the NURBS surface are preserved. Therefore, a new set of basis functions with a new set of control points is defined in a way to preserve the continuity of the surface, since the knot insertions that result in knot repetitions would reduce the continuity. We note that we use knot insertion as a refinement strategy only for the boundary surface of the domain.\\ 
\noindent Suppose $\mathcal{T}_{h_0}$ is the coarsest mesh, and define a family of meshes $\{\mathcal{T}_h\}_{h>0}$ as described in Section~\ref{sec:method}. As we insert knots to the knot spans, we update the set of Greville points and the piecewise bilinear functions that are dual to these points.
In the interior of the domain, we do the h-refinement for h values of $\frac{1}{2}$,$\frac{1}{4}$, and $\frac{1}{8}$, respectively. Over the boundary layer, we consider a uniform refinement of the boundary, that is, we insert same number of knots in each direction. Thus, each element $I_Q\in \mathcal{T}_S^p$ is respectively subdivided into $2$, $4$ and $8$ new elements that correspond to the newly defined knot spans on the two-dimensional NURBS reference domain. This enables preserving the one-to-one correspondence between interior and exterior faces of the boundary layer elements that are opposite to each other, thus, the blending function method can be used without any modifications. After the refinement via one knot insertion in every direction of the NURBS parametric domain and the discretization parameter is halved in the interior region, we have four hexahedrals each of which has a single face on the NURBS boundary (See Figure~\ref{fig:nurbsrefe}). Each one of these hexahedrals are images of the reference cube under the map $\tilde{F}_{Q}$. \\
\noindent The tensor product construction implies that a modification of a Bezier element propagates through the entire parameter space. In knot insertion, this is an important limitation as it prohibits local refinement. In numerical simulations where physical fields may change rapidly, local refinement may be required. 
Various techniques such as T-splines \cite{tspline, borden-t}, locally refined B-splines \cite{patrizi}, hierarchical B-splines 
\cite{forsey,evans2020hierarchical} and hierarchical NURBS \cite{hierNURBS} have been developed to address this issue. However, local mesh refinement is beyond the scope of this manuscript.
\begin{figure}[ht]
    \centering
    \includegraphics[scale=0.3]{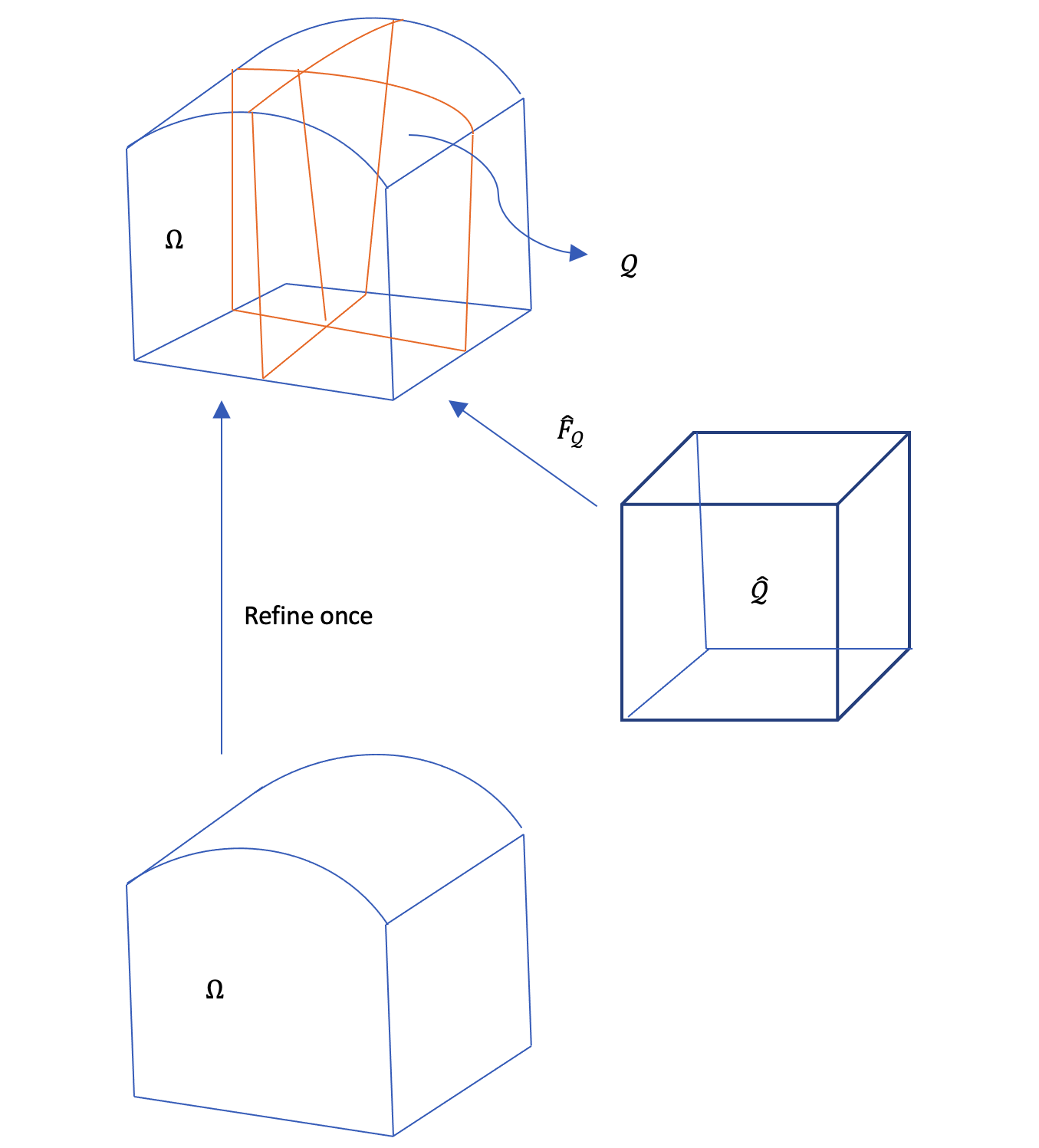}
    \caption{Refinement of a boundary layer element.}
    \label{fig:nurbsrefe}
\end{figure}
\section{Approximation properties}\label{sec:stability}
\noindent In this section, we introduce the interpolation operators and list or derive their stability and approximation properties. Interpolation error estimates play a key role in the derivation of finite element error bounds.\\
 As mentioned in the earlier sections, we focus on geometries involving a single NURBS patch, but our results can be generalized to geometries consisting of multiple patches via standard techniques.
 We employ the Lagrange interpolation operator in the interior region and construct a novel interpolant for the boundary layer by hybridizing the Lagrange interpolant with the NURBS interpolation operator defined in \cite{iga_bazilevs}. We show that our interpolants would yield optimal convergence rates.\\ \\ 
Let $h=\max\limits_{\mathcal{Q}\in \mathcal{T}_h}\{ h_{\mathcal{Q}}\}$ be the global mesh size of the family of hexahedral meshes $\{\mathcal{T}_h\}_{h>0}$ defined over
$\bar{\Omega}$ as in Section~\ref{sec:method}. We assume that $\mathcal{T}_h$ is shape-regular for every $h$, that is, the ratio of the smallest edge
of an element $\mathcal{Q}$ and its diameter,
$h_{\mathcal{Q}}$,
is uniformly bounded for $\forall \mathcal{Q}\in \mathcal{T}_h$ for every $h>0$. Thus, the ratio of the sizes of two neighboring elements is uniformly
bounded, that is, $\{\mathcal{T}_h\}_{h>0}$ is locally quasi-uniform\cite{iga_bazilevs}. 
\subsection{Interpolation Operators}
\noindent NURBS interpolation operators (and, as a result, our hybrid interpolant) are constructed from B-spline interpolation operators. Therefore, we first provide a brief overview of the B-spline interpolation operators that we will utilize. \\
\noindent Suppose $S_{p_i}(\Sigma_i)$ is the space generated by B-spline basis functions of degree $p_i$ using a knot vector $\Sigma_i$ in 1D, and let $n_i$ denote the dimension of $S_{p_i}(\Sigma_i)$. Denote by $\textbf{j}=(j_1,j_2,j_3)$ and $\textbf{m}=(m_1,m_2,m_3)$ the multi-indices that satisfy $1\le j_i,m_i\le n_i$, and $\textbf{p}=(p_1,p_2, p_3)$ the vector of polynomial degrees and let $\mathbf{\Sigma}=(\Sigma_1 \times \Sigma_2 \times \Sigma_3)$, where $\Sigma_i$ denotes the knot vector in the $i^{th}$ parametric direction. Then, we can define a 3D B-spline space as the tensor product of 1D spaces, that is, $S_{\textbf{p}}(\Sigma):=  S_{p_1}(\Sigma_1) \otimes S_{p_2}(\Sigma_2) \otimes S_{p_3}(\Sigma_3) $. Similarly, interpolation operators for $n$-dimensional B-spline spaces can be defined via the tensor product of spline interpolation operators in 1D. Therefore,  it suffices to state an explicit definition of the spline interpolation operators in 1D.\\ 
Let $\Pi_{p_i}^i: L^{2}([0,1])\to S_{p_i}(\Sigma_i)$ be a spline interpolation operator defined as $\Pi_{p_i}^i(f):=\sum\limits_{k=1}^{n_i} \lambda_{k}^{p_i}(f)B_{k}^{p_i}$, where $\lambda_{k}^{p_i}$ denotes a dual functional associated with the B-spline basis function $B_{k}^{p_i}$ and is defined as in \cite{farin}. In 3D,  for each multi-index $\textbf{j}$ defined as above, we define the dual functional associated with the B-spline basis function $\mathbf{B}_{\textbf{j}}^{\textbf{p}}:=\big(B_{j_1}^{p_1}\otimes B_{j_2}^{p_2}\otimes B_{j_3}^{p_3}\big)$ as $\mathbf{\lambda}_{\textbf{j}}^{\textbf{p}}:= \big(\lambda_{j_1}^{p_1}\otimes \lambda_{j_2}^{p_2}\otimes \lambda_{j_3}^{p_3}\big)$ such that
$\mathbf{\lambda}_{\textbf{j}}^{\textbf{p}}(\mathbf{B}_{\textbf{m}}^{\textbf{p}})=\delta_{\textbf{jm}},$
where $\delta_{\textbf{jm}}:=\delta_{j_1m_1}\delta_{j_2m_2}\delta_{j_3m_3}$ is the Kronecker delta function\cite{schumaker, veiga-buffa}. Then, the spline interpolation operator $\Pi_{\textbf{p}}$ in 3D is given by \cite{veiga-buffa, my-bgg} :
\begin{equation}\label{eq:tp-project}
    \Pi_{\textbf{p}} :=\Big(\Pi_{p_1}^1\otimes \Pi_{p_2}^2\otimes  \Pi_{p_3}^3\Big)
\end{equation}
\noindent To define the 3D NURBS interpolant that we need for the derivation of our hybrid interpolant, we first extend the parametric domain $[0,1]^2$ of the NURBS surface described in Section~\ref{sec:method} to $[0,1]^3$ by a linear extrusion along the $\zeta$-axis, and define the following NURBS space over $[0,1]^3$:
\begin{equation}\label{eq:3DNURBSpace_p}
  \mathcal{\hat{N}}_h:= span\{R_{ijk}:\ 1\le i\le n_1;\ 1\le j\le n_2 ;\ 1\le k\le n_3, \},  
\end{equation}
\noindent where $\{n_i\}_{i=1}^2$ denote the number of B-spline basis functions used in the definition of $\partial \Omega$ as before, and $n_3=2$ denotes the number of basis functions defined along the $\zeta$-direction of the 3D parametric domain $[0,1]^3$. We set the weights associated with the control points corresponding to the basis functions defined along the $\zeta$-direction equal to 1. Thus, the NURBS basis functions corresponding to these control points reduce to the B-spline basis functions that coincide with the piecewise-linear basis functions, that is, we have $R_1(\zeta)=\zeta$ and $R_2(\zeta)=(1-\zeta)$ and $\Sigma_3=\{0,0,1,1\}$.
\begin{remark}
   One may regard piecewise-linear polynomials as a special case of B-splines as can be seen in the conversion of the NURBS and B-spline bases to a Lagrangian basis in the $\zeta$-direction above. In accordance with this line of reasoning, it is possible to describe a flat surface using B-spline or NURBS basis functions, thus, represent $\Omega_{\mathcal{B}}$ as a 3D NURBS object without changing its geometry. Such an approach is implicitly employed by any CAD software that uses NURBS to represent 3D shapes with planar faces.
\end{remark}
\noindent Using the basis functions $\{R_{ijk}\}$ in \eqref{eq:3DNURBSpace_p}, we first define a boundary layer $\tilde{\Omega}_{\mathcal{B}}$ and note that 
$(\tilde{\Omega}_{\mathcal{B}}\cap \partial \Omega)$ and  $(\tilde{\Omega}_{\mathcal{B}}\cap \Omega_{int})$ are represented by identical NURBS surfaces $-$ one of which can be regarded as an offset of the other in the $\zeta$-direction. However, our boundary layer $\Omega_{\mathcal{B}}$ has the NURBS surface describe only its intersection with $\partial \Omega$. Therefore, we adjust or nullify the weights of the control points describing the offset of $\partial \Omega$, which is $\Omega_{int} \cap \tilde{\Omega}_{\mathcal{B}}$, in a way that would yield a planar interface between $\tilde{\Omega}_{\mathcal{B}}$ and $\Omega_{int}$. This results in the transformation of $\tilde{\Omega}_{\mathcal{B}}$ into 
$\Omega_{\mathcal{B}}$.\\

\noindent Then, we define the extension $S_{*}:[0,1]^3\to \Omega_{\mathcal{B}}$ of the NURBS map $S$ to 3D as follows:$$S_{*}(\alpha, \beta, \zeta)=\sum\limits_{i,j,k=1}^{n_1,n_2,n_3} R_{ijk}(\alpha,\beta,\zeta)  C_{ijk}$$
where $\{C_{ijk}\}$ denotes the set of control points used to define $\Omega_{\mathcal{B}}$ as a 3D NURBS volume. \\
We assume that $S_{*}$ is a bi-Lipschitz homeomorphism as $S$ is (See Section~\ref{sec:hybrid_mesh}).
\noindent Now, we can define the standard NURBS projection operator $\Pi_{\mathcal{\hat{N}}_h}: L^2([0,1]^3)\to \mathcal{\hat{N}}_h$ using \eqref{eq:tp-project}\cite{epfl-interpolant,veiga-buffa}:
\begin{align} \label{eq:nurbs_param_proj}
    \Pi_{\mathcal{\hat{N}}_h}v:=\frac{\Pi_{\textbf{p}}(Wv)}{W},\quad \forall v\in L^2([0,1]^3),
\end{align}
where $W$ denotes the weight function derived by using the 3D B-spline basis functions obtained from the tensor product of the B spline basis functions that span $S_1(\Sigma_3)$ and the B-spline basis functions used in the definition of $\partial \Omega$.\\
Finally, we define the NURBS space $\mathcal{N}_h$ in the physical domain as the push-forward of the NURBS space $\mathcal{\hat{N}}_h$ in \eqref{eq:3DNURBSpace_p} as follows:
\begin{equation}\label{eq:3DNURBSpace_physical}
  \mathcal{N}_h:= span\{R_{ijk}\circ S_{*}^{-1}:\ 1\le i\le n_1;\ 1\le j\le n_2 ;\ 1\le k\le n_3, \},  
\end{equation}
Then, the NURBS interpolation operator $\Pi_{\mathcal{N}_h}: L^2(\Omega_{\mathcal{B}}) \to \mathcal{N}_h$ is defined as the push-forward of $\Pi_{\mathcal{\hat{N}}_h}$\cite{epfl-interpolant,veiga-buffa}:
\begin{align}\label{eq:nurbs-interop}
    \Pi_{\mathcal{N}_h}v:=(\Pi_{\mathcal{\hat{N}}_h}(v\circ S_{*})) \circ S_{*}^{-1}.
\end{align}
\begin{remark}
  Note that $\Pi_{\mathcal{\hat{N}}_h}$ and $\Pi_{\mathcal{N}_h}$ are auxiliary operators that we define for the purpose of analysis only. We will use these operators to define the hybrid interpolation operator $\Pi_2: L^2(\Omega_{\mathcal{B}})\to V_h^b$.
\end{remark}

\begin{lemma} 
$\Pi_{\mathcal{N}_h}$ is $L^2$-stable, that is,
$\|\Pi_{\mathcal{N}_h} v \|_{L^2(\Omega_{\mathcal{B}})}  \le \|v\|_{L^2(\Omega_{\mathcal{B}})}$ for all $v\in L^2(\Omega_{\mathcal{B}})$. 
\end{lemma}

\begin{proof}
By using Lemma 3.5 in \cite{iga_bazilevs}, \eqref{eq:nurbs-interop} and referring to the $L^2$-stability of $\Pi_{\textbf{p}}$ \cite{schumaker} (and thus $\Pi_{\mathcal{\hat{N}}_h}$), we obtain
\begin{align*}
\|\Pi_{\mathcal{N}_h}v\|^2_{L^2(\mathcal{Q})} 
&\le c \|det(\nabla S_{*})\|_{L^{\infty}(\hat{\mathcal{Q}})}\|(\Pi_{\mathcal{N}_h}v)\circ S_{*}\|^2_{L^2(\hat{\mathcal{Q}})}\\
&= c \|det(\nabla S_{*})\|_{L^{\infty}(\hat{\mathcal{Q}})}\|(\Pi_{\mathcal{\hat{N}}_h}(v\circ S_{*}))\circ S_{*}^{-1}\circ  S_{*}\|^2_{L^2(\hat{\mathcal{Q}})}\\
&= c \|det(\nabla S_{*})\|_{L^{\infty}(\hat{\mathcal{Q}})}\|(\Pi_{\mathcal{\hat{N}}_h}(v\circ S_{*})\|^2_{L^2(\hat{\mathcal{Q}})}\\
&\le c \|det(\nabla S_{*})\|_{L^{\infty}(\hat{\mathcal{Q}}^{ext})}\|v\circ S_{*}\|^2_{L^2(\hat{\mathcal{Q}}^{ext})} \\
& = c \|det(\nabla S_{*})\|_{L^{\infty}(\hat{\mathcal{Q}}^{ext})} \sum\limits_{\hat{\mathcal{Q}} \in \hat{\mathcal{Q}}^{ext}}\|v\circ S_{*}\|^2_{L^2(\hat{\mathcal{Q}})}\\
& \le c \|det(\nabla S_{*})\|_{L^{\infty}(\hat{\mathcal{Q}}^{ext})} \sum\limits_{\mathcal{Q} \in \mathcal{Q}^{ext}}\|det(\nabla S_{*}^{-1})\|_{L^{\infty}(\mathcal{Q})} \|v\|^2_{L^2(\mathcal{Q})}\\
& \le c \|det(\nabla S_{*})\|_{L^{\infty}(\hat{\mathcal{Q}}^{ext})}  \|det(\nabla S_{*}^{-1})\|_{L^{\infty}(\mathcal{Q}^{ext})}\sum\limits_{\mathcal{Q} \in \mathcal{Q}^{ext}} \|v\|^2_{L^2(\mathcal{Q})}\\
& = c \|det(\nabla S_{*})\|_{L^{\infty}(\hat{\mathcal{Q}}^{ext})} \|det(\nabla S_{*}^{-1})\|_{L^{\infty}(\mathcal{Q}^{ext})}\|v\|^2_{L^2(\mathcal{Q}^{ext})}\\
& \le c\|v\|^2_{L^2(\mathcal{Q}^{ext})},
\end{align*}
where we used $\mathcal{Q}$ to denote a physical mesh element in $\mathcal{T}_h^b$ and $\hat{\mathcal{Q}}$ to denote its pre-image in the reference domain, and the superscripts \textit{ext} indicate the support extensions as in Section~\ref{sec:method}. The result follows by summing over $\mathcal{Q}\in \Omega_{\mathcal{B}}$.
\end{proof}
\begin{proposition}[Proposition 3.1\cite{epfl-interpolant}]\label{thm:epfl31}
    Given $l,s\in \mathbb{Z}$ such that $0\le l\le s\le (p+1)$ and $s\ge m$, any function $u\in H^s(\Omega_{\mathcal{B}})$ satisfies
    \begin{align*}
        \sum\limits_{\mathcal{Q}\in \mathcal{T}_{h}^b} |u- \Pi_{\mathcal{N}_h}u|^2_{H^l(\mathcal{Q})}\le Ch^{2(s-l)}\|u\|^2_{H^s(\Omega_{\mathcal{B}})}
    \end{align*}
    
\end{proposition}
   \noindent In the Proposition~\ref{thm:epfl31} above, we have adjusted the notation used in the original proposition in \cite{epfl-interpolant} to our notational setting. Note that $m=1$ in our case since we consider second-order scalar elliptic PDEs as a potential application of our method. See \cite{epfl-interpolant} for detailed apriori error estimates for the IGA-based approximations of solutions of scalar elliptic PDEs of order $2m$.\\

\noindent In the domain interior, we first define the piecewise-linear Lagrange interpolant $\Pi_{1,\mathcal{Q}}: L^2(\mathcal{Q})\to V_h^{int}(\mathcal{Q})$ for every $\mathcal{Q}\in \mathcal{T}_{h}^{(i)}$. Then, we construct its global counterpart $\Pi_1:L^2(\Omega_{int})\to V_h^{int}$ by letting $\Pi_1|_\mathcal{Q}:=\Pi_{1,\mathcal{Q}} $. After that, we define its extension $\bar{\Pi}_1$ to $L^2(\Omega)$ by using only the corner points of the NURBS surface elements as the boundary nodes of every $\mathcal{Q}\in \mathcal{T}_{h}^{(b)}$. Suppose $\Pi_1^{\mathcal{B}}$ is the restriction of $\bar{\Pi}_1$ to $L^2(\Omega_{\mathcal{B}})$. 

\noindent Finally, we define the hybrid interpolant $\Pi_2: L^2(\Omega_{\mathcal{B}})\to V_h^b$ as follows: 
$$\Pi_2:= \frac{1-\tilde{\zeta}}{1+\tilde{\zeta}} \Pi_{\mathcal{N}_h}+\frac{2\tilde{\zeta}}{1+\tilde{\zeta}}  \Pi_1^{\mathcal{B}}.$$
\noindent where $\tilde{\zeta} \in (0,1)$ is a weighting parameter that depends on the shape of the domain and is expected to lie toward the lower end of this interval as the curvature of the domain boundary increases  \\  
\noindent The $L^2$-stability of $\Pi_2$ follows from the $L^2$-stability of $\Pi_1$ and $\Pi_{\mathcal{N}_h}$.
\begin{lemma}\label{lem:boundary-inter}
 Let $v\in H^1(\Omega_{\mathcal{B}})$. Then, we have $ \|v-\Pi_2 v\|_{L^2(\Omega_{\mathcal{B}})}\le ch\|v\|_{H^1(\Omega_{\mathcal{B}})}$.
\end{lemma}
\begin{proof}
    For $\mathcal{Q}\in \mathcal{T}_h^{(b)}$, we may write
    \begin{align*}
    \|v-\Pi_2 v\|^2_{L^2(\mathcal{Q})}=& \|\frac{1-\tilde{\zeta}}{1+\tilde{\zeta}}  \Pi_{\mathcal{N}_h}v+\frac{2\tilde{\zeta }}{1+\tilde{\zeta}}\Pi_1^{\mathcal{B}}v-v\|^2_{L^2(\mathcal{Q})}\\
=&\|  \frac{1-\tilde{\zeta}}{1+\tilde{\zeta}} \Pi_{\mathcal{N}_h}v+ \frac{1-\tilde{\zeta}}{1+\tilde{\zeta}}v-\frac{1-\tilde{\zeta}}{1+\tilde{\zeta}}v+\frac{2\tilde{\zeta }}{1+\tilde{\zeta}}\Pi_1^{\mathcal{B}}v-\frac{2\tilde{\zeta }}{1+\tilde{\zeta}}v+\frac{2\tilde{\zeta }}{1+\tilde{\zeta}}v-v\|^2_{L^2(\mathcal{Q})}\\
=&\| ( \frac{1-\tilde{\zeta}}{1+\tilde{\zeta}}\Pi_{\mathcal{N}_h}v-\frac{1-\tilde{\zeta}}{1+\tilde{\zeta}}v)+(\frac{2\tilde{\zeta }}{1+\tilde{\zeta}}\Pi_1^{\mathcal{B}} v-\frac{2\tilde{\zeta }}{1+\tilde{\zeta}}v)+ \frac{1-\tilde{\zeta}}{1+\tilde{\zeta}}v+\frac{2\tilde{\zeta }}{1+\tilde{\zeta}}v-v\|^2_{L^2(\mathcal{Q})}\\
=&\|  \frac{1-\tilde{\zeta}}{1+\tilde{\zeta}}(\Pi_{\mathcal{N}_h} v-v)+\frac{2\tilde{\zeta }}{1+\tilde{\zeta}}(\Pi_1^{\mathcal{B}}v-v)\|^2_{L^2(\mathcal{Q})}\\
\le &  ( \frac{1-\tilde{\zeta}}{1+\tilde{\zeta}})^2\|\Pi_{\mathcal{N}_h}v-v\|^2+(\frac{2\tilde{\zeta }}{1+\tilde{\zeta}})^2\|\Pi_1^{\mathcal{B}} v-v\|^2_{L^2(\mathcal{Q})}\\
\le &\|\Pi_{\mathcal{N}_h}v-v\|^2+\|\Pi_1^{\mathcal{B}} v-v\|^2_{L^2(\mathcal{Q})}
\end{align*}
\color{black}
\noindent Summing over $\mathcal{Q}\in \mathcal{T}_h^{(b)}$, using Proposition~\ref{thm:epfl31} with $l=0$, $s=1$ and $p:=\min\limits_{i\in\{1,2,3\}}\{p_i\}=1$, and the approximation properties of the nodal interpolant\cite{ciarlet1972, guermond}, we obtain 
\begin{align*}
    \|v-\Pi_2 v\|^2_{L^2(\Omega_{\mathcal{B}})}
\le  c h^2 \|v\|^2_{H^1(\Omega_{\mathcal{B}})} 
\end{align*}
\end{proof}
\noindent Finally, we define the global interpolant $\Pi_h:L^2(\Omega)\to V_h$ 
\[ 
\Pi_hv(x) = 
\begin{cases} 
      \Pi_1 v(x), & \text{if } x \in \Omega_{int}, \\
      \Pi_2v(x), & \text{if } x \in \Omega_{\mathcal{B}}.
\end{cases}
\]
\noindent 
 Thus, we may write $\Pi_h := \Pi_1 \oplus \Pi_2.$\\
\noindent Letting $H^1(\Omega):=H^1(\Omega_{int})\oplus H^1(\Omega_{\mathcal{B}})$, we may write any $v\in H^1(\Omega)$ as $v:=\pi_{1} v + \pi_{2}v $, where $\pi_1$ and $\pi_2$ denote the respective projection homeomorphisms onto $H^1(\Omega_{int})$ and $H^1(\Omega_{\mathcal{B}})$ associated with the direct sum. This yields 
\begin{align*}
    \Pi_h v&= (\Pi_1\oplus \Pi_2)(\pi_1 v + \pi_2 v),\\
    &= \Pi_1(\pi_1 v) \oplus \Pi_2(\pi_2 v).
\end{align*}

\begin{corollary}\label{cor:interop-l2}
    $\Pi_h$ is $L^2$-stable on $\Omega$, and there holds $\|v-\Pi_h v\|_{L^2(\Omega)}
    \le ch\|v\|_{H^1(\Omega)}$ for all $v\in H^1(\Omega)$.
\end{corollary}
\begin{proof}
  The stability result follows from the definitions and properties of the interpolants, $\Pi_1$ and $\Pi_2$. And, by Lemma~\ref{lem:boundary-inter} and the approximation properties of the nodal interpolants \cite{brenner}, we have
\begin{align*}
\|v - \Pi_h v\|^2_{L^2(\Omega)}
&= \|v - \Pi_h v\|^2_{L^2(\Omega_{\text{int}})} + \|v - \Pi_h v\|^2_{L^2(\Omega_\mathcal{B})} \\
 &= \|v-(\Pi_1(\pi_1v)\oplus\Pi_2(\pi_2v))\|^2_{L^2(\Omega_{int})}+ \|v-(\Pi_1(\pi_1v)\oplus \Pi_2(\pi_2v))\|^2_{L^2(\Omega_{\mathcal{B}})}\\
 &=\|v - \Pi_1 v\|^2_{L^2(\Omega_{\text{int}})} + \|v - \Pi_2 v\|^2_{L^2(\Omega_\mathcal{B})} \\
&\le  c h^2 |v|^2_{H^1(\Omega_{\text{int}})}+ c h^2 \|v\|^2_{H^1(\Omega_{\mathcal{B}})}  \leq c h^2 \|v\|^2_{H^1(\Omega)}
\end{align*}

    \color{black}
\end{proof}

\begin{lemma}\label{lem:boundary-inter-h1}
 If $v\in H^2(\Omega_{\mathcal{B}})$, then we have 
 \begin{align}
     \|v-\Pi_2 v\|_{H^1(\Omega_{\mathcal{B}})}&\le ch\|v\|_{H^2(\Omega_{\mathcal{B}})},\label{eq:improved_h1_pi2}\\
      \|v-\Pi_2 v\|_{L^2(\Omega_{\mathcal{B}})}&\le ch^2\|v\|_{H^2(\Omega_{\mathcal{B}})}. \label{eq:improved_l2_pi2}
 \end{align}
\end{lemma}
\begin{proof}
Using the definition of $\Pi_2$, for every $\mathcal{Q}\in \mathcal{T}_h^{(b)}$ we may write 
\begin{align*}
|v - \Pi_2 v|^2_{H^1(\mathcal{Q})} 
&= \left|v -  \frac{1 - \tilde{\zeta}}{1 + \tilde{\zeta}}\Pi_{\mathcal{N}_h} v - \frac{2\tilde{\zeta}}{1 + \tilde{\zeta}} \Pi_1^{\mathcal{B}}  v \right|^2_{H^1(\mathcal{Q})} \\
&=\left|v - (\frac{1 - \tilde{\zeta}}{1 + \tilde{\zeta}}\Pi_{\mathcal{N}_h} v - \frac{1 - \tilde{\zeta}}{1 + \tilde{\zeta}}v)- \frac{1 - \tilde{\zeta}}{1 + \tilde{\zeta}}v-(\frac{2\tilde{\zeta}}{1 + \tilde{\zeta}} \Pi_1^{\mathcal{B}}  v -\frac{2\tilde{\zeta}}{1 + \tilde{\zeta}}v) -\frac{2\tilde{\zeta}}{1 + \tilde{\zeta}}v \right|^2_{H^1(\mathcal{Q})} \\
&=\left|  (\frac{1 - \tilde{\zeta}}{1 + \tilde{\zeta}} \Pi_{\mathcal{N}_h} v - \frac{1 - \tilde{\zeta}}{1 + \tilde{\zeta}}v)+(\frac{2\tilde{\zeta}}{1 + \tilde{\zeta}} \Pi_1^{\mathcal{B}} v -\frac{2\tilde{\zeta}}{1 + \tilde{\zeta}}v)  \right|^2_{H^1(\mathcal{Q})} \\
&\le (\frac{1 - \tilde{\zeta}}{1 + \tilde{\zeta}} )^2|\Pi_{\mathcal{N}_h} v - v|^2_{H^1(\mathcal{Q})} +(\frac{2\tilde{\zeta}}{1 + \tilde{\zeta}})^2 | \Pi_1^{\mathcal{B}}v -v|^2_{H^1(\mathcal{Q})} \\&\le |\Pi_1^{\mathcal{B}} v -v|^2_{H^1(\mathcal{Q})}+|\Pi_{\mathcal{N}_h} v - v|^2_{H^1(\mathcal{Q})}
\end{align*}

\noindent Since $v\in H^2(\Omega_{\mathcal{B}})$, we can set $l=1$, $s=2$, and $p=1$ in Proposition~\ref{thm:epfl31}. Then, summing over $\mathcal{Q}\in \mathcal{T}_h^{(b)}$ yields
\begin{align}\label{eq:interop-h1est}
    |v-\Pi_2 v|^2_{H^1(\Omega_{\mathcal{B}})}
\le  c h^2 \|v\|^2_{H^2(\Omega_{\mathcal{B}})} 
\end{align}
Using Lemma~\ref{lem:boundary-inter} and \eqref{eq:interop-h1est}, we deduce that
\begin{align*}
    \|v-\Pi_2v\|^2_{H^1(\Omega_{\mathcal{B}})}&=\|v-\Pi_2v\|^2_{L^2(\Omega_{\mathcal{B}})}+|v-\Pi_2v|^2_{H^1(\Omega_{\mathcal{B}})}\\
   & \le ch^2\big(\|v\|^2_{H^1(\Omega_{\mathcal{B}})}+\|v\|^2_{H^2(\Omega_{\mathcal{B}})}\big)\le ch^2 \|v\|^2_{H^2(\Omega_{\mathcal{B}})}
\end{align*}
This completes the proof of \eqref{eq:improved_h1_pi2}. \eqref{eq:improved_l2_pi2} is derived similarly.
\end{proof}
\color{black}
\begin{corollary}\label{cor:global-interop-est}
 If $v\in H^2(\Omega)$, then we have 
 \begin{align*}
     \|v-\Pi_h v\|_{H^1(\Omega)}&\le ch\|v\|_{H^2(\Omega)},\\
   \|v-\Pi_h v\|_{L^2(\Omega)}&\le ch^2\|v\|_{H^2(\Omega)}.
 \end{align*}
\end{corollary}

\begin{proof}
The result follows from Lemma~\ref{lem:boundary-inter-h1} and the approximation properties of the nodal interpolants\cite{brenner}. As in Corollary~\ref{cor:interop-l2}, we write
    $$ \|v-\Pi_h v\|^2_{H^1(\Omega)}\le ch^2\|v\|^2_{H^2(\Omega_{\mathcal{B}})}+ch^2|v|^2_{H^2(\Omega_{int})} \le ch^2\|v\|^2_{H^2(\Omega)} $$
    Furthermore, by Lemma~\ref{lem:boundary-inter-h1} and the standard approximation theory \cite{brenner}, we obtain
    $$ \|v-\Pi_h v\|^2_{L^2(\Omega)}\le ch^4\|v\|^2_{H^2(\Omega_{\mathcal{B}})}+ch^4|v|^2_{H^2(\Omega_{int})} \le ch^4\|v\|^2_{H^2(\Omega)} $$
\end{proof}

\section{Model Problem}\label{sec:modelprb}
In this section, we illustrate how our method can be applied to a scalar second-order linear elliptic problem in variational form. We first discuss how  theoretical results such as discrete well-posedness and convergence can be evaluated using the finite element spaces and interpolation operators introduced in Sections~\ref{sec:method} and~\ref{sec:stability}. Then, we address the Poisson problem as a representative example.\\
\noindent Let $\Omega\subset \mathbb{R}^3$ be an open, bounded, convex domain with Lipschitz boundary
and $V:= H^1(\Omega)$. 
Consider the weak form of a scalar second-order linear elliptic PDE over $\Omega$ that reads: Find $u \in V$ such that
\begin{align} \label{eq:2ndorder-system}
    a(u,v)&=b(v), \quad \forall v\in V, 
\end{align}
where $a: V \times V\to \mathbb{R}$ is a continuous and coercive bilinear form and $b:V\to \mathbb{R}$ is a continuous linear functional associated with the variational formulation of the PDE. Thus, the problem stated by \eqref{eq:2ndorder-system} is well-posed due to the Lax-Milgram Theorem \cite{brenner}.\\

\noindent Let $\mathcal{V}_h$ be defined as in \eqref{eq:global-fespace}. 
Then, the finite element formulation of \eqref{eq:2ndorder-system} reads:
Find $u_h \in \mathcal{V}_{h}$ such that
\begin{align}\label{eq:2ndorder-system-fem}
    a(u_h,v_h)=b(v_h), \quad \forall v_h\in \mathcal{V}_h.
   \end{align}
 
\noindent As $\mathcal{V}_{h}\subseteq V$, the coercivity and the continuity of $a(\cdot,\cdot)$ over $\mathcal{V}_h$ follows from the continuous case. By adapting Theorem 3.2 in \cite{epfl-interpolant} into our framework, we derive the following result.
\begin{lemma}
Let $u\in H^2(\Omega)$ be the exact solution of the system listed in \eqref{eq:2ndorder-system} and $u_h\in\mathcal{V}_h$ be the discrete solution approximating $u$ via the finite element formulation in \eqref{eq:2ndorder-system-fem}. Then, we have
   \begin{equation*}
    \|u-u_h\|_{H^1(\Omega)}\le  c h \|u\|_{H^2(\Omega)}
\end{equation*}
\end{lemma}
\begin{proof}
By Corollary~\ref{cor:global-interop-est} and Cea's lemma~\cite{ciarlett}, we obtain the result as follows:
    \begin{equation*}
    \|u-u_h\|_{H^1(\Omega)}\le C\inf_{v_h\in \mathcal{V}_h} \|u-v_h\|_{H^1(\Omega)} \le \|u-\Pi_hu\|_{H^1(\Omega)} \le c h \|u\|_{H^2(\Omega)}
\end{equation*}
\end{proof}
\noindent The error in $L^2$-norm can be computed using the standard Aubin-Nitsche duality argument.
\noindent As an example, consider the Poisson's problem given by: 
\begin{align}\label{eq:poisson}
    -\Delta u&= f,\quad \text{in}\ \Omega,\\
    u&=0, \quad \text{on}\ \partial \Omega, \notag
\end{align}
where $\Delta$ is the Laplace operator and $f \in H^{-1}(\Omega)$. The weak form of \eqref{eq:poisson} is obtained by using \eqref{eq:poisson-form} and \eqref{eq:poisson-functional} below in \eqref{eq:2ndorder-system}
\begin{align}
   a(u,v)&:=\int_{\Omega} \nabla u\ \nabla v\ dx, \label{eq:poisson-form}\\
   b(v)&:= \int_{\Omega}f\ v\ dx, \label{eq:poisson-functional}
   \end{align}
  for all $v\in H^1_0(\Omega)$. Note that, in this case, $V=:H_0^1(\Omega)$, thus, $\mathcal{V}_h:=V_h^0$ where $V_h^0$ is defined as in \eqref{eq:global-fespace-bc}.
  
\subsection{Numerical Experiments}
In this section, we apply the proposed NURBS-enhanced finite element method to \eqref{eq:poisson}. Assume that $\Omega$ is the three-dimensional domain obtained by deforming the top boundary of the unit cube while keeping the other boundary faces unchanged. Thus, it is defined as follows:
\[
\Omega = \{(x,y,z)\in\mathbb{R}^3 \;|\; (x,y)\in[0,1]^2,\; 0 \le z \le z_{\mathrm{top}}(x,y)\},
\]
where $z_{\mathrm{top}}$ denotes the deformed face represented by a NURBS surface.\\
\indent Let $z_{\mathrm{top}}$ be a parabolic surface. Note that such domains are widely used in engineering applications, including flows in curved channels and transport phenomena near solid interfaces. However, we do not focus on specific applications in this manuscript. Instead, we present experimental results for our model problem to assess the robustness of the proposed method and provide supporting evidence for the theoretical results. \\ 
\indent In the experiments below, we use the $8$-point Gauss–Legendre quadrature rule for the numerical integrations over the interior elements and the version of our blended quadrature rule mentioned in Remark~\ref{rem:blended_quad_4} for the numerical integrations over the NURBS-enhanced elements.\\
In each experiment, we employ a manufactured solution $u_{exact}$ to define the source function $f$ and to ensure the satisfaction of the homogeneous Dirichlet boundary condition. Precisely, we let $z_{\text{top}}(x,y)=1 + b\,x(1-x)\,y(1-y)$, where $b$ is a curvature parameter, and define $u_{exact}$ as follows:
\[
u_{exact}(x,y,z)
=
x(1-x)\,y(1-y)
z\left(
\frac{1}{z_{\text{top}}(x,y)}
-
\frac{z}{z_{\text{top}}(x,y)^2}
\right).
\]
\noindent For brevity, we will use NEFEM-Hex to indicate the NURBS-enhanced finite element method that we propose in the rest of this section.
\FloatBarrier
\begin{figure}[H]
    \centering
    \begin{subfigure}[t]{0.48\textwidth}
        \centering
        \includegraphics[width=\textwidth]{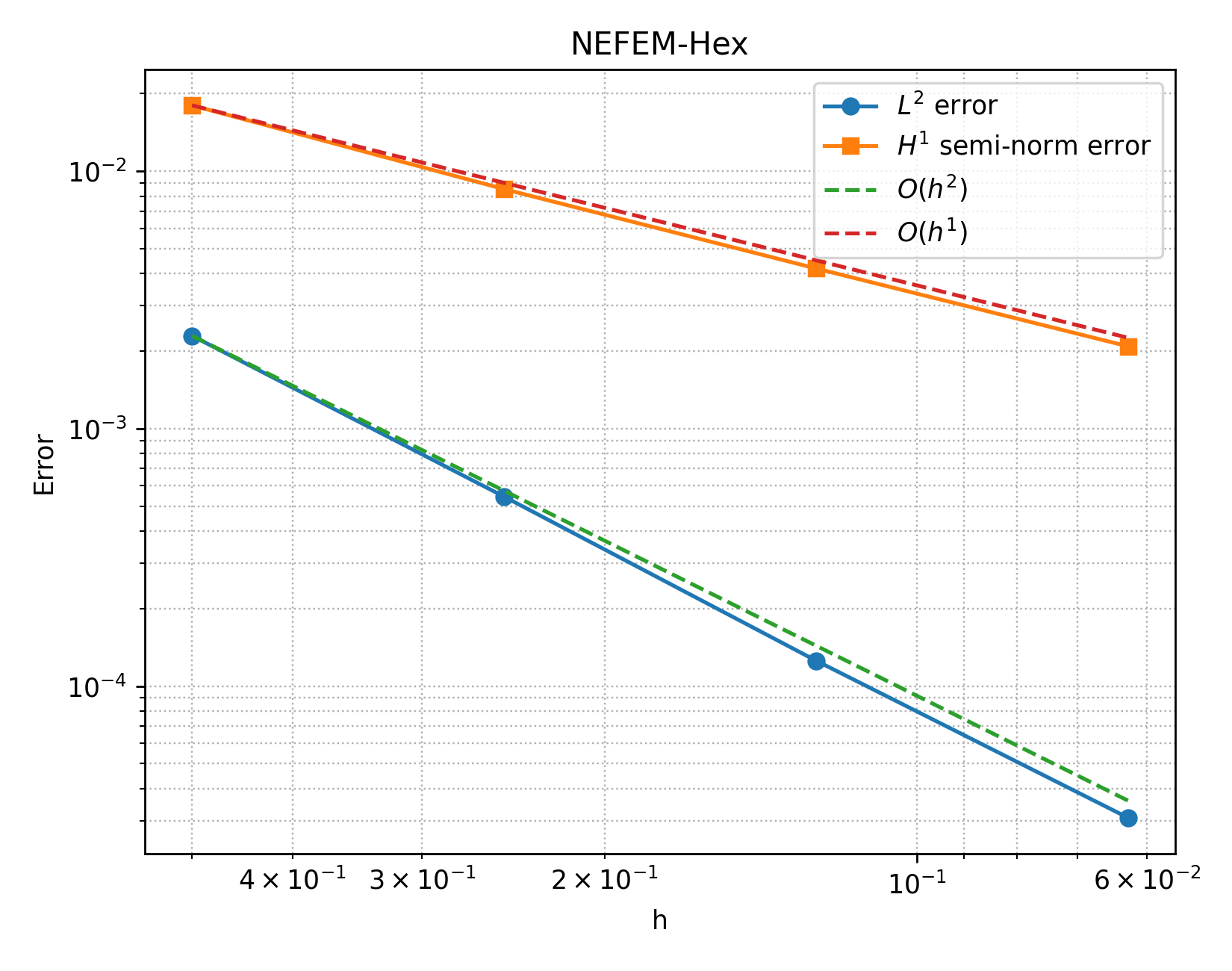}
        \caption{Flat domain, $b=0$}
        \label{fig:conv_b0}
    \end{subfigure}
    \hfill
    \begin{subfigure}[t]{0.48\textwidth}
        \centering
        \includegraphics[width=\textwidth]{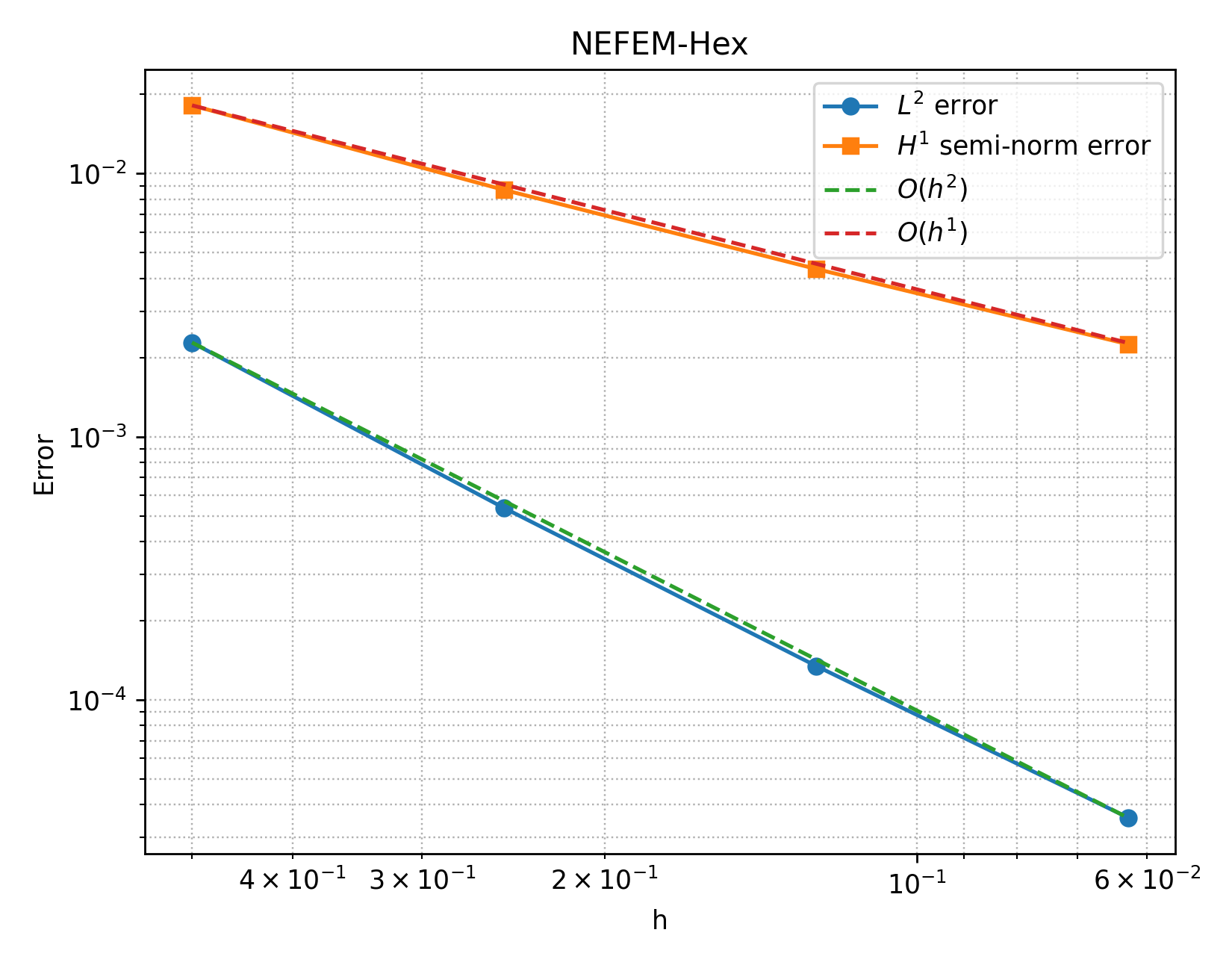}
        \caption{Mildly curved domain, $b=1$}
        \label{fig:conv_b1}
    \end{subfigure}
    \caption{Convergence rates for the Poisson problem on the unit cube with flat and mildly-curved top faces.}
    \label{fig:conv_flat_mild}
\end{figure}

\FloatBarrier
\begin{figure}[H]
    \centering
    \begin{subfigure}[t]{0.48\textwidth}
        \centering
        \includegraphics[width=\textwidth]{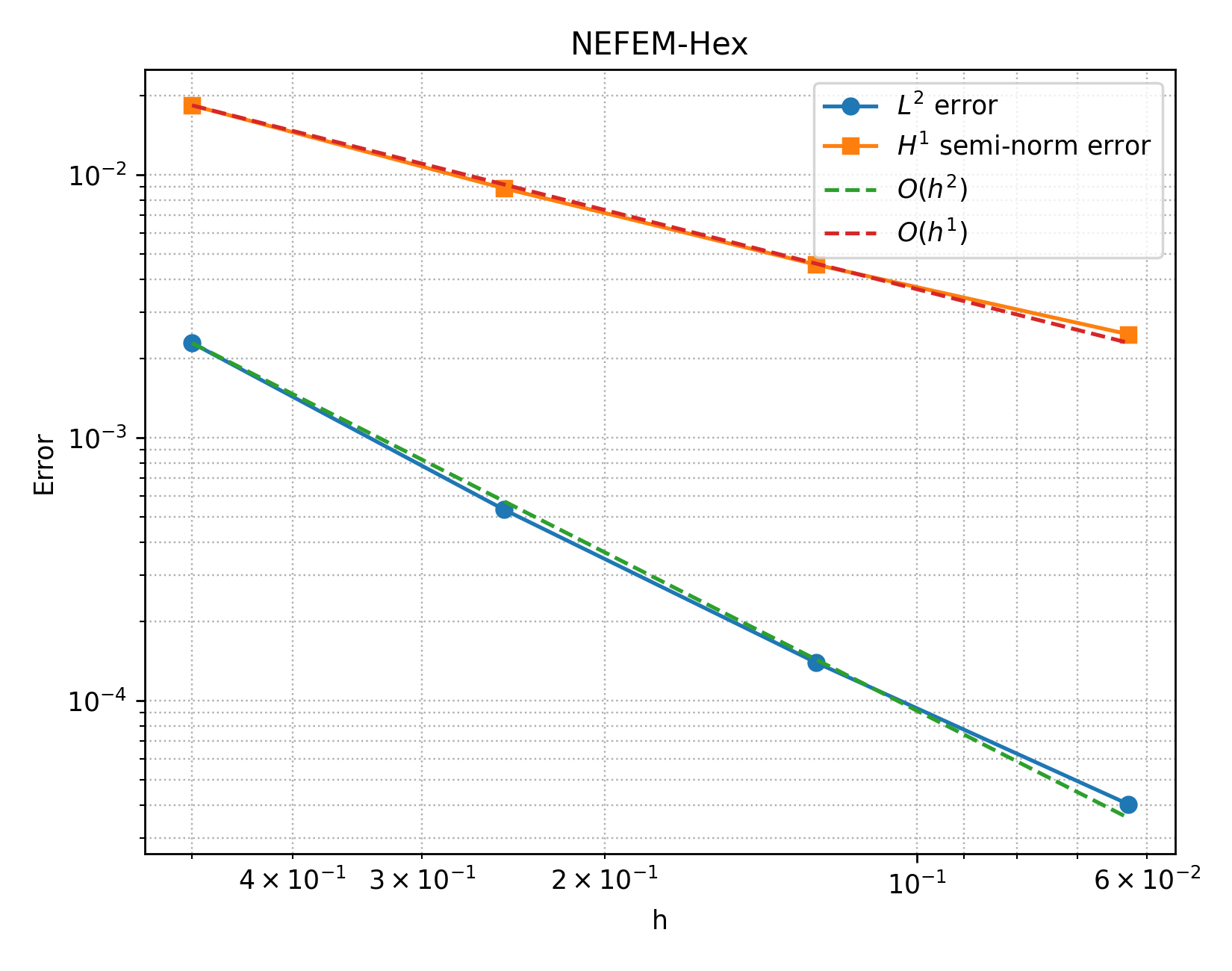}
        \caption{Moderately curved domain, $b=1.5$}
        \label{fig:conv_b0}
    \end{subfigure}
    \hfill
    \begin{subfigure}[t]{0.48\textwidth}
        \centering
        \includegraphics[width=\textwidth]{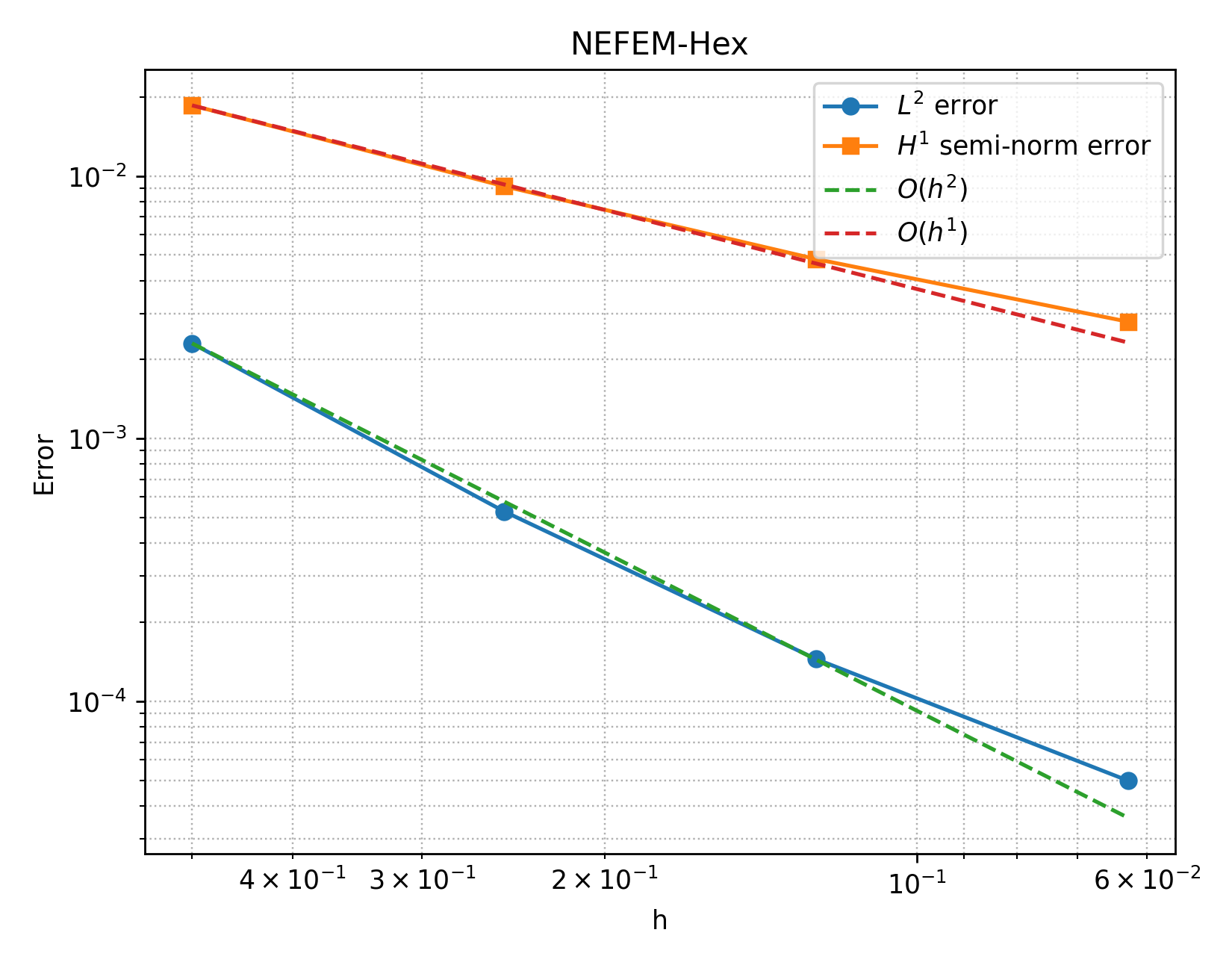}
        \caption{Strongly curved domain, $b=2$}
        \label{fig:conv_b1}
    \end{subfigure}
    \caption{Convergence rates for the Poisson problem on the unit cube with moderately and strongly-curved top faces.}
    \label{fig:conv_flat_mild}
\end{figure}
\noindent Table~\ref{tab:err_table_b0_b1} presents the approximation errors of NEFEM-Hex for the Poisson problem over the unit cube with flat and mildly-curved top faces, whereas Table~\ref{tab:err_table_b15_b2} presents the approximation errors of NEFEM-Hex for the Poisson problem over the unit cube with mildly and strongly-curved top faces. 
\FloatBarrier
\begin{table}[ht]
\centering
\setlength{\arrayrulewidth}{0.4pt} 
\renewcommand{\arraystretch}{1.1}  

\begin{tabular*}{\textwidth}{@{\extracolsep{\fill}}|c|c|c!{\vrule width 1.5pt}c|c|c|}
\hline
\multicolumn{3}{|c!{\vrule width 1.5pt}}{(i)}
& \multicolumn{3}{c|}{(ii)} \\ \hline
h & $L^2$ error & $H^1$ error & h & $L^2$ error & $H^1$ error \\ \hline
0.5000 & 2.294535e-03 & 1.801403e-02 & 0.5000 & 2.280504e-03  & 1.816409e-02 \\ \hline
0.2500 & 5.444017e-04 & 8.515521e-03  & 0.2500 & 5.361970e-04 & 8.667953e-03 \\ \hline
0.1250 & 1.256608e-04 & 4.187026e-03 & 0.1250 & 1.346189e-04 & 4.340239e-03  \\ \hline
0.0625  & 3.078114e-05 & 2.085859e-03  & 0.0625 & 3.567230e-05  & 2.246357e-03  \\ \hline
\end{tabular*}
\caption{Approximation errors for the Poisson problem over the unit cube with (i) flat and (ii) mildly-curved top faces.}
\label{tab:err_table_b0_b1}
\end{table}
\FloatBarrier
\begin{table}[ht]
\centering
\setlength{\arrayrulewidth}{0.4pt} 
\renewcommand{\arraystretch}{1.1}  
\begin{tabular*}{\textwidth}{@{\extracolsep{\fill}}|c|c|c!{\vrule width 1.5pt}c|c|c|}
\hline
\multicolumn{3}{|c!{\vrule width 1.5pt}}{(i)}
& \multicolumn{3}{c|}{(ii)} \\ \hline
h & $L^2$ error & $H^1$ error & h & $L^2$ error & $H^1$ error \\ \hline
0.5000 & 2.289051e-03 & 1.833709e-02 & 0.5000 &  2.306175e-03 & 1.856369e-02 \\ \hline
0.2500 & 5.319569e-04 & 8.864501e-03 & 0.2500 &  5.287346e-04 & 9.135427e-03  \\ \hline
0.1250 & 1.393158e-04 & 4.542358e-03 & 0.1250 &  1.449148e-04 & 4.823532e-03  \\ \hline
0.0625  & 4.032935e-05& 2.470708e-03 & 0.0625 &   4.989741e-05&  2.785766e-03\\ \hline
\end{tabular*}
\caption{Approximation errors for the Poisson problem over the unit cube with (i) moderately and (ii) strongly-curved top faces.}
\label{tab:err_table_b15_b2}
\end{table}

\subsection{Discussion}
As predicted by the theoretical analysis presented in the previous sections, the numerical experiments demonstrate that NEFEM-Hex exhibits optimal convergence both in the $H^1$ norm and in the $L^2$ norm for flat and mildly curved geometries, comparable to those the underlying $\mathbb{Q}_1$ finite element method would yield, while accurately representing the exact domain geometry. On a more strongly curved domain, the method maintains stable convergence behavior with $\mathcal{O}(h)$ convergence in the $H^1$ norm, and $\mathcal{O}(h^2)$ convergence in the $L^2$ norm with a slight deviation observed towards the finest mesh refinement.\\
\indent In addition, element-level diagnostics show that the moment-fitting system has full column rank and is satisfied with the resulting set of weights, achieving good precision. In particular, we observe that $\|Aw-b\|=\mathcal{O}(10^{-6})$. The resulting numerical solutions exhibit the same global accuracy and convergence behavior as those obtained using standard Gauss-Legendre quadrature rules. The blended quadrature rule is introduced primarily to emphasize that the integration of hybrid basis functions naturally motivates the derivation of a special quadrature rule. The systematic design of enhanced quadrature rules capable of accurately treating integrands involving differentials of hybrid basis functions is beyond the scope of this manuscript and will be addressed in future work focusing on quadrature design. \\
\indent The results support the theory and illustrate that the optimal convergence behavior of the underlying finite element method can be maintained while preserving the exact geometric representation of curved domains with the proposed method. Therefore, they confirm that exact boundary representations can be integrated into low-order finite element frameworks without compromising accuracy since the proposed method does not introduce any additional degradation that can be attributed to either the NURBS enhancement or the blended quadrature rule. The proposed method may be utilized in large-scale engineering simulations where low-order methods are still preferred due to their robustness and computational efficiency.

\section{Conclusion} \label{sec:conclusion}
\noindent In this manuscript, we proposed a NURBS-enhanced finite element method that integrates the NURBS boundary representation of a geometric domain into a standard finite element framework over hexahedral meshes. The proposed methodology combines the efficiency of finite element analysis with the geometric precision of NURBS, and may enable more accurate and efficient simulations over curved geometries.\\ 
\indent We considered the decomposition of a 3D domain with NURBS boundary into two parts: boundary layer and interior region. We defined NURBS-enhanced finite elements for the boundary layer of the domain and employed piecewise-linear Lagrange finite elements in the interior region of the domain. We introduced an interpolation operator for the NURBS-enhanced finite elements and derived its approximation properties. In addition,  we introduced a special quadrature rule for evaluating the integrals over the NURBS-enhanced finite elements. We also briefly discussed how h-refinement in finite element analysis and knot insertion in isogeometric analysis could be used in sync while preserving the hybrid finite element structure. Moreover, we described how our methodology can be applied to a generic scalar second-order linear elliptic boundary value problem and derived a priori error estimates. Finally, we provided numerical results using the Poisson problem as a model problem over curved domains, where a portion of the domain boundary was represented by a NURBS surface with varying curvature.\\ 
\indent We note that while the underlying $\mathbb{Q}_1$ finite element method relies on a piecewise planar approximation of the boundary, the proposed NURBS-enhanced finite element method employs an exact representation of the domain geometry. The numerical results demonstrate that NEFEM-Hex exhibits the expected approximation properties and convergence behavior consistent with the theoretical predictions. Overall, these results demonstrate that exact geometric representations can be incorporated into low-order finite element frameworks without introducing additional sources of error or compromising convergence. This supports the use of NEFEM-Hex in large-scale engineering simulations, where low-order methods remain attractive due to their robustness and computational efficiency. \\
\indent In future research, we will focus on the practical aspects of the proposed methodology, particularly its application to specific problems governed by second-order elliptic PDEs and the use of the proposed approach to enhance higher-order finite element methods. We will also conduct research on the adaptations of the blended quadrature rule for more general settings. Other promising research directions include extending the approach to general curved hexahedral meshes, incorporating local mesh refinement techniques, and developing strategies for handling singularities.
\section*{Acknowledgments}
The author conducted most of this research while partially funded by the Engineering and Physical Sciences Research Council (EPSRC) through grants EP/S005072/1 and EP/R029423/1. She is grateful to Nik Petrinic from the Department of Engineering Science at the University of Oxford for valuable discussions on the applications of NURBS-enhanced finite element methods and the limitations of alternative numerical approaches in solid mechanics. 

\bibliographystyle{abbrv}
\bibliography{references}{}

\end{document}